\documentclass[12pt]{amsart}
\usepackage[english]{babel}
\usepackage[utf8]{inputenc}
\usepackage{amsmath, amssymb, amsthm, amsfonts, amsxtra, stmaryrd, enumerate, graphicx}
\usepackage[colorinlistoftodos]{todonotes}
\usepackage{placeins}

\newtheorem{theorem}{Theorem}[section]
\newtheorem{corollary}{Corollary}[theorem]
\newtheorem{claim}[theorem]{Claim}
\newtheorem{subclaim}[theorem]{Subclaim}
\newtheorem{lemma}[theorem]{Lemma}
\newtheorem{fact}[theorem]{Fact}
\newtheorem{proposition}[theorem]{Proposition}

\newtheorem{question}[theorem]{Question}

\usepackage{tikz-cd}
\newtheorem{definition}{Definition}[section]

\theoremstyle{definition}

\newtheorem{remark}[theorem]{Remark}
\newtheorem{example}[theorem]{Example}

\def\Ind#1#2{#1\setbox0=\hbox{$#1x$}\kern\wd0\hbox to 0pt{\hss$#1\mid$\hss}
\lower.9\ht0\hbox to 0pt{\hss$#1\smile$\hss}\kern\wd0}

\def\ind{\mathop{\mathpalette\Ind{}}}

\def\notind#1#2{#1\setbox0=\hbox{$#1x$}\kern\wd0
\hbox to 0pt{\mathchardef\nn=12854\hss$#1\nn$\kern1.4\wd0\hss}
\hbox to 0pt{\hss$#1\mid$\hss}\lower.9\ht0 \hbox to 0pt{\hss$#1\smile$\hss}\kern\wd0}

\def\nind{\mathop{\mathpalette\notind{}}}

\usepackage{etoolbox}
\patchcmd{\subsection}{-.5em}{.5em}{}{}

\title{Properties of independence in $\mathrm{NSOP}_{3}$ theories}
\author{Scott Mutchnik}

\begin{document}

\begin{abstract}
We prove some results about the theory of independence in $\mathrm{NSOP}_{3}$ theories that do not hold in $\mathrm{NSOP}_{4}$ theories. We generalize Chernikov's work on simple and co-simple types in $\mathrm{NTP}_{2}$ theories to types with $\mathrm{NSOP}_{1}$ induced structure in $\mathrm{NDCTP}_{2}$ and $\mathrm{NSOP}_{3}$ theories, and give an interpretation of our arguments and those of Chernikov in terms of the characteristic sequences introduced by Malliaris. We then prove an extension of the independence theorem to types in $\mathrm{NSOP}_{3}$ theories whose internal structure is $\mathrm{NSOP}_{1}$. Additionally, we show that in $\mathrm{NSOP}_{3}$ theories with symmetric Conant-independence, finitely satisfiable types satisfy an independence theorem similar to one conjectured by Simon for invariant types in $\mathrm{NTP}_{2}$ theories, and give generalizations of this result to invariant and Kim-nonforking types.
\end{abstract}
\maketitle

\section{Introduction}

A central program in pure model theory is to develop the theory of independence, which originated within the stable theories, beyond stability and simplicity. This has been successful for the original notion of \textit{forking-independence} within \textit{$\mathrm{NTP}_2$ theories}: for example, Chernikov and Kaplan, in \cite{CK09}, show that forking coincides with dividing in $\mathrm{NTP}_2$ theories; Ben-Yaacov and Chernikov, in \cite{BYC07}, give an independence theorem for forking-independence in $\mathrm{NTP}_2$ theories that is improved by Simon in \cite{Sim20}, and Chernikov, in \cite{Che14}, studies simple types in $\mathrm{NTP}_2$ theories and gives a characterization of $\mathrm{NTP}_2$ theories in terms of Kim’s lemma. In a different direction, Kaplan and Ramsey in \cite{KR17} extend the original theory of independence in simple theories to \textit{$\mathrm{NSOP}_1$ theories} by introducing the notion of \textit{Kim-independence}, described as forking-independence “at a generic scale.” Kaplan and Ramsey, in \cite{KR17}, show, using work of Chernikov and Ramsey in \cite{CR15}, that symmetry of Kim-independence characterizes the property $\mathrm{NSOP}_1$; they also show that the independence theorem for Kim-independence characterizes $\mathrm{NSOP}_1$. To give examples of further consequences of $\mathrm{NSOP}_1$ for the theory of Kim-independence, Kaplan and Ramsey in \cite{KR19} give a characterization of $\mathrm{NSOP}_1$ in terms of transitivity, Kaplan, Ramsey and Shelah in \cite{KRS19} give a characterization in terms of local character; Dobrowolski, Kim and Ramsey in \cite{DKR22} and Chernikov, Kim and Ramsey in \cite{CKR20} study independence over arbitrary sets in $\mathrm{NSOP}_{1}$ theories. Kruckman and Ramsey, in \cite{KR18}, prove an improved independence theorem, developed further by Kruckman, Tran and Walsberg in the appendix of \cite{KTW22}. Kim (\cite{K21}) initiates a theory of canonical bases. For extensions to positive logic, see \cite{KD22}, \cite{K22}, \cite{DGK23}; see also \cite{CK19} for extensions of Kim-independence to $\mathrm{NTP}_{2}$ theories. Beyond $\mathrm{NSOP}_{1}$ and $\mathrm{NSOP}_{2}$, the author in \cite{NSOP2} develops a theory of independence in $\mathrm{NSOP}_2$ theories and uses this to show that every $\mathrm{NSOP}_2$ theory is in fact $\mathrm{NSOP}_1$, and Kim and Lee, in \cite{KL22}, use remarks by the author in \cite{NSOP2} to develop Kim-forking and Kim-dividing in the \textit{$\mathrm{NATP}$ theories} introduced by Ahn and Kim in \cite{AK20} and further devloped by Ahn, Kim and Lee in \cite{AK21}, as well as the related $\mathrm{NDCTP}_2$ theories introduced by the author in \cite{NSOP2}.

However, much remains to be understood about the theory of independence in Shelah's strong order hierarchy, $\mathrm{NSOP}_n$, for $n \geq 3$. In \cite{GFA}, the author relativizes the theory of Kim-independence in \cite{CR15}, \cite{KR17} by developing a theory of independence relative to abstract independence relations generalizing the free amalgamation axioms of \cite{Co15}; though the theories to which this result applies may be strictly $\mathrm{NSOP}_4$ ($\mathrm{NSOP}_4$ and $\mathrm{SOP}_3$) as well as $\mathrm{NSOP}_1$, $\mathrm{NSOP}_4$ is not actually used in the result. The author also observes in the same paper using the generalization in \cite{NSOP2} of the arguments of \cite{Co15} that theories possessing independence properties with no known $\mathrm{NSOP}_{4}$ counterxamples--symmetric Conant-independence and the strong witnessing property that generalizes Kim's lemma--cannot be strictly $\mathrm{NSOP}_{3}$. Conant-independence, which can be described as forking-independence at a \textit{maximally} generic scale and is grounded in the \textit{strong Kim-dividing} of \cite{KRS19}, is introduced in that paper (based on a similar notion with the same name developed in \cite{NSOP2} to show the equivalence of $\mathrm{NSOP}_1$ and $\mathrm{NSOP}_2$) as a potential extension of the theory of Kim-independence beyond $\mathrm{NSOP}_1$. There the author shows that a theory where Conant-independence is symmetric must be $\mathrm{NSOP}_4$, and characterizes Conant-independence in most of the known examples of $\mathrm{NSOP}_4$ theories, where it is symmetric. This leaves open the question of whether Conant-independence is symmetric in any $\mathrm{NSOP}_4$ theory, a question intimately related to the question of whether any $\mathrm{NSOP}_3$ theory is $\mathrm{NSOP}_2$. In \cite{CKR23}, Kaplan, Ramsey and Simon have recently shown that all binary $\mathrm{NSOP}_{3}$ theories are simple, by developing a theory of independence for a class of theories containing all binary theories. In \cite{SOPEXP} the author develops the independence relations $\ind^{\eth^n}$, based on the same idea of forking-independence at a maximally generic scale, shows that any theory where $\ind^{\eth^n}$ is symmetric must be $\mathrm{NSOP}_{2^{n+1}+1}$, and characterizes $\ind^{\eth^n}$ in the classical examples of  $\mathrm{NSOP}_{2^{n+1}+1}$ theories, leaving open the question of whether $\ind^{\eth^n}$ is symmetric in any $\mathrm{NSOP}_{2^{n+1}+1}$ theory. (Demonstrating robustness of the result, the author proves a similar result for left and right transitivity.) In \cite{MS17}, Malliaris and Shelah initiate a structure theory for $\mathrm{NSOP}_3$ theories, though instead of a theory of independence along the lines of forking-independence or Kim-independence, they show symmetric inconsistency for higher formulas, a result on sequences of realizations of two invariant types yielding inconsistent instances of two formulas, rather than any kind of indiscernible sequence witnessing the dividing of a single formula. Malliaris, in \cite{Mal10b}, also investigates the \textit{graph-theoretic} depth of independence in $\mathrm{NSOP}_{3}$ theories. The pressing question remains, for $n \geq 3$: using the assumption that $T$ is $\mathrm{NSOP}_n$ (and possibly some additional assumptions that are not already known to collapse $\mathrm{NSOP}_n$ into $\mathrm{NSOP}_1$), can we show any properties of $T$ that fit into the program of generalizing the properties of independence in stable or simple theories, as was done for $\mathrm{NSOP}_{1}$ and $\mathrm{NSOP}_{2}$ theories?

The aim of this paper is to show that this question is tractable for $\mathrm{NSOP}_3$ theories, whose equivalence with $\mathrm{NSOP}_1$ remains open. We prove three results on $\mathrm{NSOP}_3$ theories, two about the $\mathrm{NSOP}_1$ ``building blocks" of $\mathrm{NSOP}_3$ theories and the independence relations between them in the global $\mathrm{NSOP}_3$ structure, and one about $\mathrm{NSOP}_3$ theories with symmetric Conant-independence. All three of these results truly use $\mathrm{NSOP}_3$ in that they fail when the assumption is relaxed to $\mathrm{NSOP}_4$ (and the first two results, though both concerning the $\mathrm{NSOP}_1$ local structure, involve separate uses of $\mathrm{NSOP}_3$ in a sense that will become apparent). The properties of $\mathrm{NSOP}_{3}$ theories demonstrated by the first and third result will also appear similar to properties known or proposed for $\mathrm{NTP}_2$ theories. In contrast to this similarity, it is open (Problem 3.16 of \cite{Che14}) whether $\mathrm{NTP}_{2} \cap \mathrm{NSOP}_n$ coincides with simplicity for $n \geq 3$: a positive answer to this question would suggest that $\mathrm{NSOP}_n$ is much \textit{different} from $\mathrm{NTP}_2$.

We give an outline of the paper.

In Section 3 we generalize work of Chernikov (\cite{Che14}) on \textit{simple types} in $\mathrm{NTP}_{2}$ theories. As the property $\mathrm{NDCTP}_{2}$ is a subclass of $\mathrm{NATP}$ which is one potential solution $X$ to \cite{Kr}'s proposed analogy ``simple : $\mathrm{NTP}_{2}$ :: $\mathrm{NSOP}_{1}$ : $X$," (\cite{NSOP2}, \cite{KL22}), it is to be expected that the analogous result for ``$\mathrm{NSOP}_{1}$ types" holds for $\mathrm{NDCTP}_{2}$ theories. What is not predicted by this analogy is that the same result on $\mathrm{NSOP}_1$ local structure holds in $\mathrm{NSOP}_3$ theories. Instead of generalizing the definition of simple types, we introduce a definition schema for the \textit{internal} properties of a (partial) type, which is more natural in that it refers to the global properties of a structure associated with that type. (We could also have generalized the definition of simple types to $\mathrm{NSOP}_{1}$ and gotten the same conclusion; see Remark \ref{4-qe}.) We show that just as Chernikov implicitly showed in \cite{Che14} for internally simple types in $\mathrm{NTP}_{2}$ theories, the assumption of $\mathrm{NSOP}_{3}$ controls how internally $\mathrm{NSOP}_{1}$ types relate to the rest of the structure:

\begin{theorem}\label{4-main0}
    Let $T$ be $\mathrm{NSOP}_{3}$, and $p(x)$ an internally $\mathrm{NSOP}_{1}$ type. Then $p(x)$ is co-$\mathrm{NSOP}_{1}$.
\end{theorem}

See Definitions \ref{4-consop1} and \ref{4-internally}. When $T$ is only assumed to be $\mathrm{NSOP}_{4}$, we give an internally simple type $p(x)$ for which this fails.

We then interpret the proof of this result as well the results of Chernikov in \cite{Che14} (and their direct generalization to $\mathrm{NDCTP}_{2}$) in terms of the \textit{characteristic sequences} introduced by Malliaris in \cite{Mal10} to relate ``classification-theoretic properties" of a theory to the ``graph-theoretic properties" of hypergraphs, and used by Malliaris in \cite{Mal12} to study Keisler's order. Internally to a type $p(x)$, what the ambient theory perceives to be an instance of co-$\mathrm{SOP}_{1}$ (an instance of $\mathrm{SOP}_{1}$ with parameters realizing $p(x)$) is simply a definable hypergraph making no reference to consistency. Model-theoretic properties of a theory will give control of the graph-theoretic structure of hypergraphs definable in that theory, similarly to Shelah's classic result that an definable bipartite graph with the order property in an $\mathrm{NSOP}$ theory must even have the independence property. Applied in the case where the model-theoretic properties, such as simplicity and $\mathrm{NSOP}_{1}$, are assumed of the internal structure on $p(x)$, this will illuminate the proof in \cite{Che14} of co-simplicity in $\mathrm{NTP}_{2}$ theories and our proof of co-$\mathrm{NSOP}_{1}$ in $\mathrm{NDCTP}_{2}$ and $\mathrm{NSOP}_{3}$ theories.

In Section 4, we discuss how internally $\mathrm{NSOP}_{1}$ types interrelate within the ambient structure of a $\mathrm{NSOP}_{3}$ theory, showing that their behavior is similar to how they would interrelate in a globally $\mathrm{NSOP}_{1}$ theory. By the Kim-Pillay characterization of $\mathrm{NSOP}_{1}$, Theorem 9.1 of \cite{KR17}, for no reasonable notion of independence could a full independence theorem hold in an $\mathrm{SOP}_{1}$ (that is, non-$\mathrm{NSOP}_{1}$) theory. However, we prove an independence theorem between internally $\mathrm{NSOP}_{1}$ types in $\mathrm{NSOP}_{3}$ theories (here, $a \subset p(\mathbb{M})$ means that $a$ is a tuple obtained by concatenating tuples realizing $p$ whose coordinates are in $\mathbb{M}$, as the standard model-theoretic notation would suggest):

\begin{theorem}\label{4-main00}
    Let $T$ be $\mathrm{NSOP}_{3}$, and let $p_{1}, p_{2}, p_{3}$ be internally $\mathrm{NSOP}_{1}$ types over $M$. Let $a_{1} \equiv_{M} a'_{1} \subset p_{1}(\mathbb{M})$, $ a_{2} \subset p_{2}(\mathbb{M})$,  $ a_{3} \subset p_{3}(\mathbb{M})$.  If $a_{1} \ind^{K^{*}}_{M} a_{2}$, $a'_{1} \ind^{K^{*}}_{M} a_{3}$, $a_{2} \ind^{K^{*}}_{M} a_{3}$, there is some $a''_{1}$ with $a''_{1} \models \mathrm{tp}(a_{1}/Ma_{2}) \cup \mathrm{tp}(a'_{1}/Ma_{3})$. Moreover, $a''_{1}$ can be chosen with $a_{2}a_{3} \ind^{K^{*}}_{M}a''_{1}$, $a_{2}a''_{1} \ind^{K^{*}}_{M} a_{3}$ and $a_{3}a''_{1} \ind^{K^{*}}_{M} a_{2}$.
\end{theorem}

Here $\ind^{K^{*}}$ is Conant-independence, Definition \ref{4-conantindependence}. Motivating this result, in an $\mathrm{NSOP}_{1}$ theory, Conant-independence coincides with Kim-independence, $\ind^{K}$ and is symmetric; compare \cite{KR17}, Theorem 6.5, which characterizes $\mathrm{NSOP}_{1}$. (Between tuples of realizations of two co-$\mathrm{NSOP}_{1}$ types $p_{i}, p_{j}$ it coincides with Kim-diving independence.) While in proving this result, we apply Theorem \ref{4-main0}, it does not just follow from co-$\mathrm{NSOP}_{1}$: we exhibit internally stable types $p_{1}, p_{2}, p_{3}$ in an $\mathrm{NSOP}_{4}$ theory $T$ for which this fails. This independence theorem for internally $\mathrm{NSOP}_{1}$ types in $\mathrm{NSOP}_{3}$ theories is not only of interest to the program of extending the theory of independence beyond $\mathrm{NSOP}_{1}$ theories; it is also of interest to the question of whether $\mathrm{NSOP}_{3}$ coincides with $\mathrm{NSOP}_{2} = \mathrm{NSOP}_{1}$. One potential approach to building a strictly $\mathrm{NSOP}_{3}$ theory (that is, one that is $\mathrm{SOP}_{2}$) is by starting with $\mathrm{NSOP}_{2}$ structures and somehow combining them to obtain a failure of $\mathrm{NSOP}_{1}$ in the form of a failure of the independence theorem: this result says that it is impossible to obtain an $\mathrm{NSOP}_{3}$ theory from such a construction. It may be of interest to ask whether there is any connection between this result on stability-theoretic independence and Theorem 7.7 of \cite{Mal10b}, which concerns the graph-theoretic depth of independence in $\mathrm{NSOP}_{3}$ theories.

In Section 5, we consider $\mathrm{NSOP}_{3}$ theories where Conant-independence is symmetric. It is natural to assume this, as there is no known $\mathrm{NSOP}_{4}$ theory where Conant-independence or Conant-dividing independence is not symmetric. Simon, in \cite{Sim20}, proves an improved independence theorem for $\mathrm{NTP}_{2}$ theories, Fact \ref{4-indntp2}. There Simon also poses an existence question, Question \ref{4-simon?}, for invariant types with the same Morley sequence in $\mathrm{NTP}_{2}$ theories (i.e. $M$-invariant types $p$ and $q$ with $p^{(\omega)}|_{M} = q^{(\omega)}|_{M}$). An independence theorem for forking-independence, for invariant types with the same Morley sequence in $\mathrm{NTP}_{2}$ theories, would follow from a positive answer to this question, by Simon's result. In an $\mathrm{NSOP}_{3}$ theory with symmetric Conant-independence, we prove a similar independence theorem for Conant-independence between \textit{finitely satisfiable} types with the same Morley sequence:

\begin{theorem} \label{4-main000}
Let $T$ be an $\mathrm{NSOP}_{3}$ theory, and assume $\ind^{K^{*}}$ is symmetric. Suppose $p$ and $q$ are $M$-finitely satisfiable (global) types with $p^{(\omega)}|_{M} = q^{(\omega)}|_{M}$, and let $a, b \supseteq M$ be small supersets of $M$ with $a \ind^{K}_{M} b$. Then there is $c \models p(x)|_{a} \cup q(x)|_{b}$ with $c \ind^{K^{*}}_{M} ab$.

\end{theorem}

This fails when $T$ is the model companion of triangle-free graphs, which is $\mathrm{NSOP}_{4}$ with symmetric (indeed trivial) Conant-independence. We also give an extension of this result from finitely satisfiable types to Kim-nonforking types when Conant-dividing independence is symmetric, which has the advantage of exploiting the full force of symmetry for Conant-independence. While this result is again of interest to the question of extending the theory of independence beyond $\mathrm{NSOP}_{1}=\mathrm{NSOP}_{2}$, since there is precedent (see \cite{NSOP2}) for using facts about independence to prove the equivalence of classification-theoretic dividing lines, it is also of interest to another open question, whether an $\mathrm{NSOP}_{3}$ theory with symmetric Conant-independence is $\mathrm{NSOP}_{1}$.

One final remark: in Theorem 3.15 of \cite{GFA}, a strategy was suggested for proving the equivalence of $\mathrm{NSOP}_{3}$ and $\mathrm{NSOP}_{2}$ by proving two facts that have no known $\mathrm{NSOP}_{4}$ counterexamples, symmetry for Conant-independence and the strong witnessing property, for all $\mathrm{NSOP}_{4}$ or even all $\mathrm{NSOP}_{3}$ theories. The results of this paper suggest a different approach, via finding properties of independence in $\mathrm{NSOP}_{3}$ theories that distinguish them from $\mathrm{NSOP}_{4}$ theories.

\section{Preliminaries}

Notations are standard. We will need some basic definitions and facts about some standard relations between sets, as well as some facts about $\mathrm{NSOP}_{1}$ and $\mathrm{NSOP}_{3}$ theories.

\textbf{Relations between sets}

Adler, in \cite{A09}, defines some properties of abstract ternary relations $A \ind_{M} B$ between sets. In our case, we will assume $M$ is a model, and we will only need to refer to a few of these properties by name:

Left extension: If $A\ind_{M} B $ and $A \subseteq C$, there is some $B' \equiv_{A} B$ with  $C\ind_{M} B' $.

Right extension: If $A\ind_{M} B $ and $B \subseteq C$, there is some $A' \equiv_{B} A$ with  $A'\ind_{M} C $.

Symmetry: If $A \ind_{M} B$ then $B \ind_{M} A$.

Chain condition with respect to invariant Morley sequences: If $A \ind_{M} B$ and $I = \{B_{i}\}_{i < \omega}$ is an invariant Morley sequence over $M$ (see below) with $B_{0} = B$, then there is $I' \equiv_{MB} I$ indiscernible over $MA$ with $A \ind_{M} I'$.

We will refer to various relations between sets. For the convenience of the reader, here is an index of the notation to be used. Kim-independence and Kim-dividing, as well as Conant-independence and Conant-dividing, will be defined later in this section.

$a \ind_{M}^{i} b$ if $\mathrm{tp}(a/Mb)$ extends to a global $M$-invariant type

$a \ind^{u}_{M} b$ if $\mathrm{tp}(a/Mb)$ extends to a global $M$-finitely satisfiable type

$a \ind^{f}_{M} b$ if $a$ is forking-independent from $b$ over $M$

$a \ind^{K}_{M} b$ if $a$ is Kim-independent from $b$ over $M$

$a \ind^{K^{*}}_{M} b$ if $a$ is Conant-independent from $b$ over $M$

$a \ind^{Kd}_{M} b$ if $\mathrm{tp}(a/Mb)$ contains no formulas Kim-dividing over $M$

$a \ind^{K^{*}d}_{M} b$ if $\mathrm{tp}(a/Mb)$ contains no formulas Conant-dividing over $M$.

We will use $\ind^{K^{+}}$ and $\ind^{K^{+}u}$ as ad-hoc notations in proofs; these will be defined in the course of those proofs.

We give an overview of some basic definitions. A \textit{global type} $p(x)$ is a complete type over the sufficiently saturated model $\mathbb{M}$. For $M \prec \mathbb{M}$, a global type $p(x)$ is \textit{invariant} over $M$ if $\varphi(x, b) \in p(x)$ and $b' \equiv_{M} b$ implies $\varphi(x, b') \in p(x)$. One class of types invariant over $M$ is the class of types that are \textit{finitely satisfiable} over $M$, meaning any formula in the type is satisfied by some element of $M$. We say an infinite sequence $\{b_{i}\}_{i \in I}$, is an \textit{invariant Morley sequence} over $M$ (in the type $p(x)$) if there is a fixed global type $p(x)$ invariant over $M$ such that $b_{i} \models p(x)|_{M\{b_{j}\}_{j < i}}$ for $i \in I$. If $p(x)$ is finitely satisfiable over $M$, we say $\{b_{i}\}_{i \in I}$ is a \textit{coheir Morley sequence} or \textit{finitely satisfiable Morley sequence} over $M$. Invariant Morley sequences over $M$ are indiscernible over $M$, and the  EM-type of an invariant Morley sequence over $M$ depends only on $p(x)$. For $p(x), q(y)$ $M$-invariant types, $p(x) \otimes q(y)$ is defined so that $ab \models p(x)\otimes q(y)|_{A}$ for $M \subseteq A$ when $b \models q(y)|_A$ and $a \models p(x)|_{Ab}$; $p(x) \otimes q(y)$ is then invariant. Similarly, for $I$ a linearly ordered set and $p(x)$ an $M$-invariant type, $p^{(I)}(x)$ is the $M$-invariant type such that $\{a_{i}\}_{i \in I} \models p^{(I)}(x)|_{A}$ for $M \subseteq A$ when, for all $i \in I$, $a_{i} \models p(x)|_{A\{a_{j}\}_{j < i}}$.

Both $\ind^{i}$ and $\ind^{u}$ have right extension, but it is sometimes advantageous to work with coheir Morley sequences rather than general invariant Morley sequences because $\ind^{u}$ is also known to have left extension.

\textbf{Kim-independence and $\mathrm{NSOP}_{1}$}

We assume knowledge of basic simplicity theory and the definition of forking-independence. An extension of the theory of independence from simple theories to \textit{$\mathrm{NSOP}_{1}$ theories} was developed by Kaplan and Ramsey in \cite{KR17}, via the definition of \textit{Kim-independence}.

\begin{definition}\label{3-nsop1}
A theory $T$ is $\mathrm{NSOP}_{1}$ if there does not exist a formula $\varphi(x, y)$ and tuples $\{b_{\eta}\}_{\eta \in 2^{<\omega}}$ such that $\{\varphi(x, b_{\sigma \upharpoonleft n})\}_{n \in \omega}$ is consistent for any $\sigma \in 2^{\omega}$, but for any $\eta_{2} \unrhd \eta_{1} \smallfrown \langle 0\rangle$, $\{\varphi(x, b_{\eta_{2}}), \varphi(x, b_{\eta_{1} \smallfrown \langle 1\rangle})\}$ is inconsistent. Otherwise it has $\mathrm{SOP}_{1}$.
\end{definition}

\begin{definition}(\cite{KR17})
A formula $\varphi(x, b)$ \emph{Kim-divides} over $M$ if there is an invariant Morley sequence $\{b_{i}\}_{i \in \omega}$ over $M$ starting with $b$ (said to \emph{witness} the Kim-dividing) such that $\{\varphi(x, b_{i})\}_{i \in \omega}$ is inconsistent. A formula  $\varphi(x, b)$ \emph{Kim-forks} over $M$ if it implies a (finite) disjunction of formulas Kim-dividing over $M$. We write $a \ind^{K}_{M} b$, and say that $a$ is \emph{Kim-independent} from $b$ over $M$ if $\mathrm{tp}(a/Mb)$ does not include any formulas Kim-forking over $M$.
\end{definition}

Kim-independence in $\mathrm{NSOP}_{1}$ theories behaves, in many ways, like forking-independence in simple theories.

\begin{fact}\label{4-kimslemma}
(\cite{KR17}) Let $T$ be $\mathrm{NSOP}_{1}$. Then for any formula $\varphi(x,b)$ Kim-dividing over $M$, any invariant Morley sequence over $M$ starting with $b$ witnesses Kim-dividing of $\varphi(x, b)$ over $M$. Conversely, suppose that for any formula $\varphi(x,b)$ Kim-dividing over $M$, any invariant Morley sequence (even in a finitely satisfiable type) over $M$ starting with $b$ witnesses Kim-dividing of $b$ over $M$. Then $T$ is $\mathrm{NSOP}_{1}$.

It follows that Kim-forking coincides with Kim-dividing in any $\mathrm{NSOP}_{1}$ theory.
\end{fact}

\begin{fact}\label{4-symmetry}
(\cite{CR15}, \cite{KR17}) The theory $T$ is $\mathrm{NSOP}_{1}$ if and only if $\ind^{K}$ is symmetric.
\end{fact}

The \textit{independence theorem} for Kim-independence in $\mathrm{NSOP}_{1}$ theories generalizes that of \cite{KP99} for simple theories, which in turn generalizes stationarity of forking-independence (the uniqueness of nonforking-extensions) in stable theories. Part of our argument for the results of Section 4 will require re-proving the independence theorem in the context of co-$\mathrm{NSOP}_{1}$ types. For motivation, we give the original statement:

\begin{fact}\label{4-independence}(\emph{Independence theorem, \cite{KR17}.})
    Let $T$ be $\mathrm{NSOP}_{1}$. Then if   $a_{1} \ind^{K}_{M} b_{1}$, $a_{2} \ind^{K}_{M} b_{2}$, $b_{1} \ind^{K}_{M} b_{2}$, and $a_{1} \equiv_{M} a_{2}$, there is some $a \ind^{K}_{M} b_{1}b_{2}$ with $a \equiv_{Mb_{i}}a_{i}$ for $i=1, 2$.
\end{fact}

\textbf{Conant-independence}

\textit{Conant-independence} was introduced in a modified form in \cite{NSOP2} to show that $\mathrm{NSOP}_{2}$ theories were $\mathrm{NSOP}_{1}$. The standard version was defined in \cite{GFA}, based on Conant's implicit use of the concept in \cite{Co15} to classify modular free amalgamation theories. It was proposed by the author of this paper as an extension of Kim-independence beyond $\mathrm{NSOP}_{1}$ theories.

\begin{definition}\label{4-conantindependence}
Let $M$ be a model and $\varphi(x, b)$ a formula. We say $\varphi(x, b)$ \emph{Conant-divides} over $M$ if for \emph{every} invariant Morley sequence $\{b_{i}\}_{i \in \omega}$ over $M$ starting with $b$, $\{\varphi(x, b_i)\}_{i \in \omega}$ is inconsistent. We say $\varphi(x, b)$ \emph{Conant-forks} over $M$ if and only if it implies a disjunction of formulas Conant-dividing over $M$. We say $a$ is \emph{Conant-independent} from $b$ over $M$, written $a \ind^{K^{*}}_{M}b$, if $\mathrm{tp}(a/Mb)$ does not contain any formulas Conant-forking over $M$.
\end{definition}

In \cite{GFA} it is shown that if Conant-independence is symmetric in a theory $T$, $T$ is $\mathrm{NSOP}_{4}$. In the same paper, Conant-independence is characterized for most of the known examples of $\mathrm{NSOP}_{4}$ theories, where it is shown to be symmetric. It is open whether Conant-independence is symmetric in all $\mathrm{NSOP}_{4}$ theories, or even all $\mathrm{NSOP}_{3}$ theories. It is also open whether theories with symmetric Conant-independence display the classification-theoretic behavior characteristic of theories with a good notion of free amalgamation, first studied in \cite{Co15} and later improved upon in \cite{GFA}: either $\mathrm{NSOP}_{1}$ or $\mathrm{SOP}_{3}$, and either simple or $\mathrm{TP}_{2}$.

\textbf{Classification theory}

In this paper, we will be interested in $\mathrm{NSOP}_{3}$ theories and how they differ from $\mathrm{NSOP}_{4}$ theories:

\begin{definition}
Let $n \geq 3$. A theory $T$ is $\mathrm{NSOP}_{n}$ (that is, does not have the \emph{n-strong order property}) if there is no definable relation $R(x_{1}, x_{2})$ with no $n$-cycles, but with tuples $\{a_{i}\}_{i \in \omega}$ with $\models R(a_{i}, a_{j})$ for $i <j$. Otherwise it has $\mathrm{SOP}_{n}$.
\end{definition}

We will need nothing about $\mathrm{NSOP}_{4}$, other than that the below counterexamples to our results on $\mathrm{NSOP}_{3}$ theories are $\mathrm{NSOP}_{4}$, because they are free amalgamation theories; see \cite{Co15}, Theorem 4.4. We will need the following syntactic fact about $\mathrm{NSOP}_{3}$, essentially proven by Shelah (\cite{She95}, Claim 2.19):

\begin{fact}\label{4-sop3fact}
    Suppose there is an array $\{a_{i}, b_{i}\}_{i < \omega}$ and formulas $\varphi(x, y)$, $\psi(x, z)$ with

    (1) For $m < n$, $\{\varphi(x, b_{i})\}_{i \leq m} \cup \{\psi(x, a_{i})\}_{m < i \leq n}$ is consistent.

    (2) For $ i < j$, $\{\varphi(x, a_{i}), \psi(x, b_{j})\}$ is inconsistent.

    Then $T$ has $\mathrm{SOP}_{3}$.
\end{fact}

Finally $\mathrm{NTP}_{2}$ and $\mathrm{NDCTP}_{1}$ will play a secondary role in this paper, but we will discuss some results on these classes that motivate our main results on $\mathrm{NSOP}_{3}$ theories.

\begin{definition}
A theory $T$ is $\mathrm{NTP}_{2}$ (that is, does not have the \emph{tree property of the second kind}) if there is no array $\{b_{ij}\}_{i, j \in \omega}$ and formula $\varphi(x, y)$ such that there is some fixed $k$ such that, for all $i$, $\{\varphi(x, b_{ij})\}_{j \in \omega}$ is $k$-inconsistent, but for any $\sigma \in \omega^{\omega}$, $\{\varphi(x, b_{i\sigma(i)})\}_{i \in \omega}$ is consistent.
\end{definition}

The class $\mathrm{NATP}$ was introduced in \cite{AK20} and further developed in \cite{AK21} as a generalization of $\mathrm{NTP}_{2}$; it has been proposed as one possible answer to a question of Kruckman \cite{Kr}, on what class can be viewed to generalize properties of $\mathrm{NIP}$ and $\mathrm{NSOP}_{1}$ theories the same way $\mathrm{NTP}_{2}$ theories generalize properties of $\mathrm{NIP}$ and simple theories. It is still open to what extent the analogy holds; for example, whether Kim-forking coincides with Kim-dividing in $\mathrm{NATP}$ theories, as forking coincides with dividing in $\mathrm{NTP}_{2}$ theories. However, for \textit{$\mathrm{NDCTP}_{2}$ theories}, introduced in \cite{NSOP2} and further developed in \cite{KL22} and \cite{Hanson23} (where $\mathrm{NDCTP}_{2}$ is referred to as the negation of the ``comb tree property", or $\mathrm{NCTP}$), the equivalence of Kim-forking and Kim-dividing was proven in \cite{KL22} after being proven for coheir Kim-dividing and coheir Kim-forking in \cite{NSOP2}.

\begin{definition}\label{4-descendingcomb}(Proposition 2.51, item IIIa, [102]). A list $\eta_{1}, \ldots, \eta_{n} \in \omega^{<\omega}$ is a \emph{descending comb} if and only if it is an antichain such that $\eta_{1} <_{\mathrm{lex}} \ldots <_{\mathrm{lex}} \eta_{n}$, and such that, for $1 \leq k < n$,  $\eta_{1} \wedge \ldots \wedge \eta_{k+1} \lhd \eta_{1} \wedge \ldots \wedge \eta_{k}.$
    
\end{definition}

\begin{definition}
The theory $T$ has \emph{$\mathrm{DCTP}_{2}$} if there exists $k < \omega$, a formula $\varphi(x, y)$ and tuples $\{b_{\eta}\}_{\eta \in 2^{<\omega}}$ such that $\{\varphi(x, b_{\sigma \upharpoonleft n})\}_{n \in \omega}$ is $k$-inconsistent for any $\sigma \in 2^{\omega}$, but for any descending comb $\eta_{1} \ldots, \eta_{l} \in 2^{< \omega}$, $\{\varphi(x, b_{\eta_{i}})\}_{i=1}^{l}$ is consistent. If $T$ does not have $\mathrm{DCTP}_{2}$, it has \emph{$\mathrm{NDCTP}_{2}$}.
\end{definition}

\section{Reflection principles for hypergraph sequences}

Simple types were defined in \cite{HKP00}; then co-simple and $\mathrm{NTP}_{2}$ types were defined in \cite{Che14}. We define \textit{co-$\mathrm{NSOP}_{1}$} types and give some equivalent definitions, similarly to Definition 6.7 of \cite{Che14}. (When clear from context, when $p(x)$ is an $n$-type we refer to $p(\mathbb{M}^{n})$ by $p(\mathbb{M})$).

\begin{definition}\label{4-consop1}
A partial type $p(x)$ over $M$ is \textit{co-$\mathrm{NSOP}_{1}$} if it satisfies one of the following equivalent conditions:

(1) There does not exist a formula $\varphi(x, y) \in L(M)$ and tuples $\{b_{\eta}\}_{\eta \in 2^{<\omega}}$, $b_{\eta} \subset p(\mathbb{M})$ such that $\{\varphi(x, b_{\sigma \upharpoonleft n})\}_{n \in \omega}$ is consistent for any $\sigma \in 2^{\omega}$, but for any $\eta_{2} \unrhd \eta_{1} \smallfrown \langle 0\rangle$, $\{\varphi(x, b_{\eta_{2}}), \varphi(x, b_{\eta_{1} \smallfrown \langle 1\rangle})\}$ is inconsistent.

(2, 2') There does not exist a formula $\varphi(x, y) \in L(M)$ and an array $\{c_{i, j}\}_{i = 0, 1, j< \omega}$, $c_{i, j} \subset p(\mathbb{M})$, such that $\{\varphi(x, c_{0, j})\}_{j < \omega}$ is consistent, $\{\varphi(x, c_{1, j})\}_{j < \omega}$ is $k$-inconsistent for some $k$ ($2$-inconsistent), and $c_{0, j} \equiv_{Mc_{0, < j}c_{1, < j}} c_{1, j}$ for each $j < \omega$.

(3) Kim's lemma for Kim-dividing: For $M' \succeq M$, and $\varphi(x, y) \in L(M)$, if $\varphi(x, b)$ Kim-divides over $M$ for $b \subset p(\mathbb{M})$, then for every $M'$-invariant Morley sequence $\{b_{i}\}_{i \in \omega}$ with $b_{0}=b$, $\{\varphi(x, b_{i})\}_{i \in \omega}$ is inconsistent.

\end{definition}

\begin{proof}
     This is essentially proven in Chernikov and Ramsey (\cite{CR15}) and Kaplan and Ramsey (\cite{KR17}) so we will only give a sketch.

    (1 $\Leftrightarrow$ 2 $\Leftrightarrow$ 2') Follows from the proof of Proposition 2.4 of \cite{CR15}  uses the proof of Proposition 5.6 of \cite{CR15}. The part due to \cite{CR15} shows that if there is a formula $\varphi(x, y) \in L(M)$ and tuples $\{b_{\eta}\}_{\eta \in 2^{<\omega}}$, $b_{\eta} \subset p(\mathbb{M})$ such that $\{\varphi(x, b_{\sigma \upharpoonleft n})\}_{n \in \omega}$ is consistent for any $\sigma \in 2^{\omega}$, but for any $\eta_{2} \unrhd \eta_{1} \smallfrown \langle 0\rangle$, $\{\varphi(x, b_{\eta_{2}}), \varphi(x, b_{\eta_{1} \smallfrown \langle 1\rangle})\}$ is inconsistent, then there is an array $\{c_{i, j}\}_{i = 0, 1, j< \omega}$, such that $\{\varphi(x, c_{0, j})\}_{j < \omega}$ is consistent, $\{\varphi(x, c_{1, j})\}_{j < \omega}$ is $2$-inconsistent, and $c_{0, j} \equiv_{Mc_{0, < j}c_{1, < j}} c_{1, j}$ for each $j < \omega$. The other direction due to \cite{KR17} shows that if there is $\varphi(x, y)$ and $\{c_{i, j}\}_{i = 0, 1, j< \omega}$, such that $\{\varphi(x, c_{0, j})\}_{j < \omega}$ is consistent, $\{\varphi(x, c_{1, j})\}_{j < \omega}$ is $k$-inconsistent for some $k$, and $c_{0, j} \equiv_{Mc_{0, < j}c_{1, < j}} c_{1, j}$ for each $j < \omega$, then there is a formula $\psi(x, y') \in L(M)$ (obtained as a conjunction of instances of $\varphi$) and tuples $\{b_{\eta}\}_{\eta \in 2^{<\omega}}$, $b_{\eta} \subset p(\mathbb{M})$ such that $\{\psi(x, b_{\sigma \upharpoonleft n})\}_{n \in \omega}$ is consistent for any $\sigma \in 2^{\omega}$, but for any $\eta_{2} \unrhd \eta_{1} \smallfrown \langle 0\rangle$, $\{\psi(x, b_{\eta_{2}}), \psi(x, b_{\eta_{1} \smallfrown \langle 1\rangle})\}$ is inconsistent. And if each $b_{\eta} \subseteq p(\mathbb{M})$ for $p(x)$ a fixed partial type, the former direction even shows that we can choose $c_{i, j} \subseteq p(\mathbb{M})$, and vice versa for the latter direction, proving the equivalence in the co-$\mathrm{NSOP}_{1}$ case.

    (3 $\Rightarrow$ 2). This is basically the proof of Proposition 3.14 of \cite{KR17}. Assume (2) is false; we show (3) is false. The equivalence (1 $\Leftrightarrow$ 2) does not use anything about the fact that $M$ is a model, and the failure of (1) to hold is preserved under expanding the language; therefore, we can fix a Skolemization $T^{\mathrm{Sk}}$ of $T$, and assume that $c_{0, j} \equiv^{L^{\mathrm{Sk}}}_{Mc_{0, < j}c_{1, < j}} c_{1, j}$ for each $j < \omega$. By Ramsey's theorem and compactness, we can choose $\{\bar{c}_{j}\}_{j < \omega}$ $M$-indiscernible in $T^{\mathrm{Sk}}$. Let $M' = \mathrm{dcl}_{\mathrm{Sk}}(\bar{c}_{<\omega}M) $. Choose non-principal ultrafilters $U_{i}$, $i = 0, 1$, containing $c_{i,<\omega}$, and for $i = 0, 1$ define the global type $p_{i}(x)$ to be equal to $\{\varphi(x, c): \varphi(M, c) \in U_{i}\}$, so that each of the $p_{i}$ are finitely satisfiable over $M$. It can be shown from $c_{0, j} \equiv^{L^{\mathrm{Sk}}}_{Mc_{0, < j}c_{1, < j}} c_{1, j}$ that $p_{0}|_{M} = p_{1}|_{M}$; let $b$ realize this, so $b \subseteq p(\mathbb{M})$. Then $\{\varphi(x, b_{i})\}_{i <\omega}$ will be consistent for $\{b_{i}\}_{i < \omega}$ a Morley sequence in $p_{0}$, but $\{\varphi(x, b'_{i})\}_{i <\omega}$ will be inconsistent for $\{b'_{i}\}_{i < \omega}$ a Morley sequence in $p_{1}$, so Kim's lemma fails.

    (2 $\Rightarrow$ 3). This is the proof of Proposition 3.15 of \cite{KR17}. We assume that (3) is false and show that (2) is false. Let $\varphi(x, b)$ for $b \subseteq p(\mathbb{M})$ Kim-divide over $M'$, winessed by a Morley sequence in the $M'$-invariant type $p_{1}$. Let Morley sequences in the $M'$-invariant type $p_{0}$ fail to witness Kim-dividing of $\varphi(x, b)$ over $M'$. Find $\{c_{0, i}, c_{1, i}\}_{i \in \mathbb{Z}}$ so that $(c_{0, i}, c_{1, i})_{i \in \mathbb{Z}} \models (p_{0} \otimes p_{1})^{ (\mathbb{Z})}$. Then $\{c_{0, i}, c_{1, i}\}_{0\leq i < \omega }$ will be as desired.

\end{proof}

While Chernikov (\cite{Che14}, Definition 6.7) gives an additional characterization of co-simplicity in terms of symmetry for forking-independence, it requires additional elements of the base to belong to $p(\mathbb{M})$. Giving a characterization of co-$\mathrm{NSOP}_{1}$ types in terms of symmetry for Kim-independence would be more complicated, because defining Kim-independence over arbitrary sets, rather than models, requires additional considerations; see \cite{KR17}. However, co-$\mathrm{NSOP}_{1}$ types over $M$ do have a symmetry property over $M$, which will be useful in the sequel:

\begin{proposition}\label{4-symm3}
(Symmetry) Let $p(x)$ be a co-$\mathrm{NSOP}_{1}$ type over $M$ and $a \subset p(\mathbb{M})$. If $a \ind^{Kd}_{M'} b$, then $b\ind_{M}^{Kd} a$.

\end{proposition}

\begin{proof}
    Because $b$ is not necessarily contained in $ p(\mathbb{M})$ construction of the original tree must proceed like the proof of Theorem 5.4 of \cite{NSOP2} (based in turn on the proof of Lemma \ref{4-consistency5b} of \cite{KR17}) in taking a specially chosen Morley sequence at each stage, rather than directly following the proof of Theorem 6.5 of \cite{KR17}. Though the rest of the proof can be done as in Lemma \ref{4-wit5} and Proposition 5.13 of \cite{KR17}, we give our own exposition, which only requires us to construct a tree of countable size rather than a much larger tree.

    Suppose $p(\mathbb{M})$ is co-$\mathrm{NSOP}_{1}$. We begin with the following claim (which we could have avoided by following Lemma \ref{4-wit5} and Proposition 5.13 of \cite{KR17}):

    \begin{claim}\label{4-bound}
        Let $\varphi(x, c)$ Kim-divide over $M$ for $\varphi(x, y) \in L(M)$ and $c \subseteq p(\mathbb{M})$. Then there is a bound depending only on $\varphi(x, y)$ and $\mathrm{tp}(c/M)$ on the size of a set $\{c_{i}\}^{n}_{i= 0}$, $c_{i} \models \mathrm{tp}(c/M)$ for $0 \leq i \leq n$, such that there are $M$-finitely satisfiable types $p_{0} \ldots p_{n}$ such that $c_{i} \models p_{i}(x)|_{Mc_{0} \ldots c_{i-1}}$ for $0 \leq i \leq n$, and $\{\varphi(x, c_{i})\}^{n}_{i = 0}$ is consistent.
    \end{claim}

    \begin{proof}
        This proceeds as in the direction (2 $\Rightarrow$ 3) of the previous definition (again, see \cite{KR17}, Proposition 3.15). Let $p_{0}, \ldots, p_{n}$ be as in the claim, and let Morley sequences in the $M$-invariant type $q(x) \vdash \mathrm{tp}(c/M)$ witness Kim-dividing of $\varphi(x, c)$ over $M$: for $\bar{c}' \models q^{(\omega)}|_{M}(x)$, $\{\varphi(x, c'_{i})\}_{i < \omega}$ is $k$-inconsistent for some fixed $k$. Find $\{c_{0, i}, c_{1, i}\}_{0 \leq i \leq n}$ such that $(c_{0, i}, c_{1, i})_{0 \leq i \leq n} \models (p_{n} \otimes q)\otimes \ldots \otimes (p_{0} \otimes q)$. Then $c_{0, n} \ldots c_{0, 0} \equiv_{M} c_{0}, \ldots c_{n} $, so $\{\varphi(x, c_{0, i})\}^{n}_{i = 0}$ is consistent. However, $c_{1, n}, \ldots, c_{1, 0} \models q^{(n +1)}|_{M}(x)$, so  $\{\varphi(x, c_{1, i})\}^{n}_{i = 0}$ is $k$-inconsistent. Finally, $c_{0, 0}c_{1,0} \ldots c_{0, i-1} c_{1, i-1} \ind^{i}_{M} c_{0, i}c_{1, i}$ for all $1 \leq i \leq n$, and $c_{0,i} \equiv_{M} c \equiv_{M} c_{1, i}$, so $c_{0,i}\equiv_{M_{c_{0, 0}c_{1,0} \ldots c_{0, i-1} c_{1, i-1}}} c_{1, i}$ for all $1 \leq i \leq n$.

        Now if $n$ is unbounded ($k$ is fixed) this contradicts Definition \ref{4-consop1} (2), by compactness.
    \end{proof}

     The following step, where we construct a tree, is where we must deviate from the proof of Theorem 5.16 of \cite{KR17}. We use the notation $a\ind^{K^{+}u}_{M} b$ to denote that there is a coheir Morley sequence $\{b_{i}\}_{i < \omega}$ over $M$ with $b_{0} = b$ that remains indiscernible over $Ma$. We prove some basic facts about this relation:

    \begin{claim}\label{4-rext}
        Right extension: The relation $\ind^{K^{+}u}$ satisfies right extension: if $a \ind^{K^{+}u}_{M} b$, for any $c$ there is some $a' \equiv_{Mb} a$ with  $a' \ind^{K^{+}u}_{M} bc$.
    \end{claim}
    \begin{proof}
        Let $I=\{b_{i}\}_{i < \omega}$ be a Morley sequence in the $M$-finitely satisfiable type $q(x)$, $b_{0} = b$, that remains indiscernible over $Ma$. By left extension for $\ind^{u}$ there is some $M$-finitely satisfiable type $r(x, y)$ extending $q(x)$ and $\mathrm{tp}(bc/M)$. Then there are $c_{i}'$, $i < \omega$ $c_{0} = c$, such that $\{b_{i}c_{i}\}_{i<\omega}$ is a Morley sequence in $r(x, y)$. By Ramsey's theorem, compactness and an automorphism, $a'$ can then be chosen such that $a' \equiv_{Mb} a$, indeed such that $a'\equiv_{MI} a$, and $\{b_{i}c_{i}\}_{i<\omega}$ is indiscernible over $Ma$.
    \end{proof}

    \begin{claim}\label{4-cc3}
        Chain condition: Let $I=\{b_{i}\}_{i \in \omega}$ be an $M$-finitely satisfiable Morley sequence indiscernible over $Ma$. Then $a \ind^{K^{+}u}_{M} I$.
    \end{claim}

\begin{proof}
    By compactness there is $I'=\{b_{i}\}_{i \in \omega^{2}}$ such that $I'|_{\omega} = I$ and $I'$ is indiscernible over $Ma$. Then $\{b'_{i \omega} \ldots b'_{i \omega+j} \ldots \}_{i < \omega}$ will be an $M$-finitely satisfiable Morley sequence starting with $I'$ and indiscernible over $Ma$.
\end{proof}

Assume for contradiction that $a \ind^{Kd}_{M} b$ with $b \subseteq p(\mathbb{M})$, but $b \nind^{Kd}_{M} a$. We find, for all $n$, a tree $(I_{n}, J_{n})= (\{a_{\eta}\}_{\eta \in \omega^{\leq n}}, \{b_{\sigma}\}_{\sigma \in \omega^{n}})$, with the first $n+1$ levels $I_{n}$ forming an infinitely branching tree, then with each $a_{\sigma}$ for $\sigma \in \omega^{n}$ at level $n+1$ followed by a single additional leaf $b_{\sigma}$ at level $n+2$, with the following two properties:

(1) For $\eta \unlhd \sigma$, $|\sigma| =n$, $a_{\eta}b_{\sigma} \equiv_{M} ab$.

(2) For $\eta \in \omega^{< n}$, the subtrees at $\eta$ (i.e. the sets of elements indexed by the cones strictly above $\eta$, ordered as $\{a_{\unrhd \eta \smallfrown \langle n \rangle}\}_{n \in \omega}$) form an $M$-finitely satisfiable Morley sequence indiscernible over $a_{\eta}$ (so for $I$ this sequence of subtrees, $a_{\eta} \ind^{K^{+}u}_{M} I$).

For $n =0$, let $a_{\emptyset}=a$, $b_{\emptyset} =b$; then (2) follows from the fact that $\ind^{Kd}$ easily implies $\ind^{K^{+}u}$. Assume $(I_{n}, J_{n})$ already constructed; we construct $(I_{n+1}, J_{n+1})$. We see by (2) that for $(I^{*}_{n}, J_{n})$ the nodes of the tree excluding $a_{\emptyset}$, $a_{\emptyset}\ind^{K^{+}u}_{M} I^{*}_{n} J_{n}$. By Claim \ref{4-rext}, find $a'_{\emptyset}\equiv_{MI^{*}_{n}J_{n}}a_{\emptyset}$ with $a'_{\emptyset} \ind^{K^{+}u}_{M} I_{n} J_{n}$, which will be the new root of $(I_{n+1}, J_{n+1})$. Then find some $M$-finitely satisfiable Morley sequence $\{(I_{n},J_{n})^{i}\}_{i \in \omega}$ starting with $(I_{n},J_{n})$ indiscernible over $Ma'_{\emptyset}$, giving the subtrees of $(I_{n}, J_{n})$ at $a'_{\emptyset}$. From $a_{\emptyset}J_{n} \equiv_{M} a'_{\emptyset}J_{n} \equiv_{M} a'_{\emptyset}J^{i}_{n}$, we will preserve (1) by indexing accordingly, and from choice of $\{(I_{n},J_{n})^{i}\}_{i \in \omega}$, we will preserve (2) as well.

We now find a contradiction to Definition 3.1.2; this is where, by constructing a much larger tree, we could have just followed Lemma \ref{4-wit5} and Proposition 5.13 of \cite{KR17}. By (1), the paths of each $I_{n}$ are consistent: for $\sigma \in \omega^{n}$, $\{\varphi(x, a_{\eta})\}_{\eta \unlhd \sigma}$ is consistent, realized by $b_{\sigma}$. But by (2), for any $k$ nodes $\eta_{1}, \ldots, \eta_{k} \in \omega^{<\omega}$, forming an antichain such that $\eta_{1} <_{\mathrm{lex}} \ldots <_{\mathrm{lex}} \eta_{k}$, and such that, for $1 \leq i < k$, $\eta_{1}\wedge \ldots \wedge \eta_{i+1} \lhd \eta_{1} \wedge \ldots \wedge \eta_{i}$, $\{a_{\eta_{i}}\}_{i = 1}^{k}$ form a sequence with $a_{\eta_{i}} \ind^{i}_{M} a_{\eta_{1}} \ldots a_{\eta_{i-1}}$; by (1), $a_{i} \subseteq p(\mathbb{M})$. So for $k$ the bound from Claim \ref{4-bound}, and $\eta_{i}$ with these conditions (forming a \textit{descending comb}, Definition \ref{4-descendingcomb}), for $\eta_{i}$ for $1 \leq i \leq k$ satisfying the above property, $\{\varphi(x, a_{\eta_{i}})\}_{i = 1}^{k}$ is inconsistent. So by compactness, we can find a tree $\{a_{\eta}\}_{\eta \in \omega^{< \omega}}$ with the same consistency and inconsistency properties for $\varphi(x, y)$ (consistency along the paths and inconsistency on descending combs of size $k$), and with $b_{\eta} \subseteq p(\mathbb{M})$.

We recall the following definition and fact:

\begin{definition}
   (Definitions 11 and 12, \cite{TT12}) For tuples $\overline{\eta}, \overline{\eta}' \in \omega^{<\omega}$ of elements of $\omega^{<\omega}$, we write $\overline{\eta} \sim_{0} \overline{\eta}'$ to mean that $\overline{\eta}$ has the same quantifier-free type in the language $\{<_{\mathrm{lex}}, \lhd, \wedge \}$ as $\overline{\eta}'$. For $(b_{\eta})_{\eta \in \omega^{<\omega}}$ a tree-indexed set of tuples, all of the same length, and $\overline{\eta} = \eta_{1}, \ldots, \eta_{n} \in \omega^{<\omega}$ an $n$-tuple of elements of $\omega^{<\omega}$, we write $b_{\overline{\eta}}=: b_{\eta_{1}}\ldots b_{\eta_{n}}$, and call $(b_{\eta})_{\eta \in \omega^{<\omega}}$ \emph{strongly indiscernible} over a set $A$ if for all tuples  $\overline{\eta}, \overline{\eta}' \in \omega^{<\omega}$ of elements of $\omega^{<\omega}$ with $\overline{\eta} \sim_{0} \overline{\eta}'$, $b_{\overline{\eta}} \equiv_{A} b_{\overline{\eta}'}$.
\end{definition}

\begin{fact} \label{4-indisc3} (Theorem 16, \cite{TT12}; see \cite{Sc15} for an alternate proof) Let $(b_{\eta})_{\eta \in \omega^{<\omega}}$ be a tree-indexed set of tuples, all of the same length, and $A$ a set. Then there is $(c_{\eta})_{\eta \in \omega^{<\omega}}$ strongly indiscernible over $A$ such that for any tuple $\overline{\eta} \in \omega^{< \omega}$ of elements of $\omega^{<\omega}$ and $\varphi(x) \in L(A)$, if $\models \varphi(b_{\overline{\eta}'})$ for all $\overline{\eta}' \sim_{0} \eta$, then $\models \varphi(c_{\overline{\eta}})$.
\end{fact}

Now use Fact \ref{4-indisc3} to extract a strongly indiscernible tree $(c_{\eta})_{\eta \in \omega^{<\omega}}$. Let $\{c_{j, i}\}_{j = 0, 1, i < \omega} =\{c_{\langle 0\rangle^{i} \smallfrown \langle j \rangle}\}$. Then $\{c_{j, i}\}_{j = 0, 1, i < \omega}$ is as in Definition 3.1.2, contradiction.

\end{proof}

We could likely have also proven Proposition \ref{4-symm3} in the style of Definition 6.1 of \cite{Che14}: use right extension to find an $\ind^{Kd}$-Morley sequence of $a$ over $M$, indiscernible over $b$, and then developed local character and Kim's lemma for $\ind^{Kd}$-Morley sequences in the context of co-$\mathrm{NSOP}_{1}$ types, following \cite{KRS19} and \cite{KR19}. Since these characterizations of co-$\mathrm{NSOP}_{1}$ are not necessary for our main theorem on internally $\mathrm{NSOP}_{1}$ types, we leave the details to the reader.

Notions such as co-simple and co-$\mathrm{NSOP}_{1}$ types involve interaction of the types with the rest of the structure. In the other direction, there are the simple types defined in \cite{HKP00}, the $\mathrm{NIP}$ and $\mathrm{NTP}_{2}$ types defined in \cite{Che14}, and the fully stable types defined in \cite{sim12}. We introduce a new schema for defining the local classification-theoretic properties of a type, which is in some sense more natural, because it depends only on the corresponding properties for a structure associated with the type.

\begin{definition}\label{4-internally}
    (1) Let $p(x)$ be a partial $n$-type over $M$. Let $\mathcal{L}_{p}$ contain an $m$-ary relation symbol $R_{\varphi}$ for each formula $\varphi(x_{1}, \ldots, x_{m}) \in L(M)$ with $|x_{i}| = n$ for $i \leq n$.\footnote{Note that this is dependent on the model $M$ over which $p(x)$ is viewed as a partial type, so in this definition $p(x)$ is taken not just as a set of formulas, but as a set of formulas together with a model $M$ containing the parameters of those formulas.} Then $\mathcal{M}_{p}$ is the $\mathcal{L}_{p}$-structure with domain $p(\mathbb{M}^{n})$ and with $R_{\varphi}(p(\mathbb{M}^{n})^{m}) = \varphi(\mathbb{M}^{mn}) \cap p(\mathbb{M}^{n})^{m}$.

    (2) Let $\mathcal{P}$ be a property of theories. Then a partial type $p(x)$ is \emph{internally $\mathcal{P}$} if the theory of $\mathcal{M}_{p}$ is $\mathcal{P}$.
\end{definition}

\begin{remark}\label{4-qe}
    If $p(x)$ is not just a partial type, but a formula with parameters in $M$, then the theory of $\mathcal{M}_{p}$ clearly has quantifier elimination. In this case, for $p$ to be internally simple, $\mathrm{NIP}$, etc. is weaker than for it to be simple or $\mathrm{NIP}$ in the sense of \cite{HKP00}, \cite{Che14}. The case of partial types given by single formulas is in fact all we need to find counterexamples in $\mathrm{NSOP}_{4}$ theories to our results on the internally $\mathrm{NSOP}_{1}$ types of $\mathrm{NSOP}_{3}$ theories. In the case of general partial types, we could have have also considered the case where $\mathcal{P}$ is a property of formulas and all quantifier-free formulas of $\mathcal{L}_{p}$ have property $\mathcal{P}$. This definition would also be weaker than the corresponding ``external" property, and our results should go through even assuming only the quantifier-free version, by developing the theory of Kim-independence relative to only the quantifier-free formulas.
\end{remark}

Theorem 6.17 of \cite{Che14} says that simple types are co-simple; in fact, only internal simplicity is needed. By way of analogy, internally $\mathrm{NSOP}_{1}$ types are co-$\mathrm{NSOP}_{1}$ in $\mathrm{NDCTP}_{2}$ theories; see below. Beyond this analogy, we find that:

\begin{theorem}\label{3-main3}
    Let $T$ be $\mathrm{NSOP}_{3}$, and $p(x)$ an internally $\mathrm{NSOP}_{1}$ type. Then $p(x)$ is co-$\mathrm{NSOP}_{1}$.
\end{theorem}

\begin{proof}
    
Again following the arguments of Chernikov and Ramsey \cite{CR15} and Kaplan and Ramsey \cite{KR17}, we start by carrying out the arguments of Definition 3.1, (1 $\Leftarrow$ 2' $\Leftarrow$ 3 (for $2$-Kim-dividing)) \textit{internally} to $\mathcal{M}_{p}$. Since the consistency in the definition of (1) need not be witnessed by a realization of $p(x)$, we will no longer be dealing with actual consistency or inconsistency of instances of $\varphi(x, y)$, but rather the definable relations in $\mathcal{M}_{p}$ corresponding to this consistency, treated only as a definable hypergraph. This hypergraph will be part of the \textit{characteristic sequence} of $\varphi(x, y)$, introduced by Malliaris in \cite{Mal10}.

Suppose $p(x)$ is not co-$\mathrm{NSOP}_{1}$. Let $\varphi(x, y) \in L(M)$, $b_{\eta} \subseteq p(\mathbb{M})$, $\eta \in 2^{\omega}$ be as in Definition 3.1.1. By compactness, we can replace $2^{< \omega}$ with $2^{<\kappa}$, for large $\kappa$. Define $R_{n}(y_{1}, \ldots, y_{n}) =: R_{\exists x \varphi(x, y_{1}) \wedge \ldots \wedge \varphi(x, y_{n})}(y_{1}, \ldots, y_{n}) \in \mathcal{L}_{p}$. Then  $\mathcal{M}_{p} \models R_{n} (b_{\eta_{1}}, \ldots, b_{\eta_{n}}) $ for $\eta_{1} \lhd \ldots \lhd \eta_{n} \in 2^{< \kappa}$, but for any $\eta_{2} \unrhd \eta_{1} \smallfrown \langle 0\rangle$, $\mathcal{M}_{p} \models \neg R_{2}(b_{\eta_{2}}, b_{\eta_{1} \smallfrown \langle 1\rangle})$.

For a sequence of relations $\{R_{n}\}_{n < \omega}$ on a set, where $R_{n}$ is an $n$-ary relation, call a sequence $\{a_{i}\}_{i \in I}$ a \textit{clique} if for $i_{1}, \ldots, i_{n} \in I$, $(a_{1}, \ldots a_{i_{n}}) \in R_{n}$, and an $n$-\textit{anticlique} if for distinct $i_{1}, \ldots, i_{n} \in I$, $(a_{1}, \ldots a_{i_{n}}) \notin R_{n}$. Choose a Skolemization of $\mathcal{M}_{p}$. We show that there is an array $\{c_{i, j}\}_{i = 0, 1, j< \omega}$, $c_{i, j} \in \mathcal{M}_{p}$, such that $\{ c_{0, j}\}_{j < \omega}$ is a clique, $\{c_{1, j}\}_{j < \omega}$ is a $2$-anticlique, and $c_{0, j} \equiv^{\mathcal{L}^{\mathrm{Sk}}_{p}}_{c_{0, < j}c_{1, < j}} c_{1, j}$ for each $j < \omega$. We may follow the proof of \cite{CR15}, Proposition 5.6, that we cited in the direction (1$\Rightarrow$ 2') of Definition 3.1. We sketch the argument: we will draw the $c_{i,j}$ from $\{b_{\eta}\}_{\eta \in \omega^{\kappa}}$. Suppose that, for $1 \leq i \leq n$ $c_{i, 0}=b_{\lambda_{i}}$, and $c_{i, 1}=b_{\eta_{i}}$ are already chosen to satisfy these properties, with $\eta_{j} \wedge \lambda_{j} \rhd \lambda_{i}$ and $\lambda_{i} \unrhd  (\eta_{i} \wedge \lambda_{i})\smallfrown \langle 0 \rangle, \eta_{i} \unrhd  (\eta_{i} \wedge \lambda_{i})\smallfrown \langle 1 \rangle $, for $1 \leq i < j \leq n$. Then using the pigeonhole principle, choose nodes $\lambda_{n+1} =\lambda_{n}\smallfrown \langle 0 \rangle^{\kappa_{1}} \smallfrown \langle 1\rangle$, $\eta_{n+1} = \lambda_{n}\smallfrown \langle 0 \rangle^{\kappa_{2}} \smallfrown \langle 1\rangle$ for $\kappa_{1} < \kappa_{2} < \kappa$ such that $c_{n+1, 0}=c_{\lambda_{n+1}}$ and $c_{n+1, 1}=c_{\eta_{n+1}}$ are such that $c_{n+1, 0} \equiv^{\mathcal{L}^{\mathrm{Sk}}_{p}}_{c_{0, \leq n}c_{1, \leq n}} c_{n+1, 1}$.

We next find a model $\mathcal{M}$ of the theory of $\mathcal{M}_{p}$ and $\mathcal{M}$-invariant Morley sequences $\{b_{i}\}_{i < \omega}$ in the $\mathcal{M}$-invariant type $p_{0}$ and $\{b'_{i}\}_{i < \omega}$ in the $\mathcal{M}$-invariant type $p_{1}$, such that $b_{0}= b'_{0}$, $\{b_{i}\}_{i < \omega}$ is a clique, and $\{b'_{i}\}_{i < \omega}$ is a $2$-anticlique.\footnote{It was observed by Hyoyoon Lee, Byunghan Kim, and the other participants of the Yonsei University logic seminar that the proof of Proposition 3.14 of \cite{KR17} actually shows that in a $\mathrm{SOP}_{1}$ theory, there is a formula that $2$-Kim-divides for which Kim's lemma fails. This is the ``internal" version of this observation.} As in the proof of (3 $\Rightarrow$ 2) of Definition 3.1, we follow the proof of Proposition 3.14 of \cite{KR17}. By Ramsey's theorem and compactness, we can choose $\{\bar{c}_{j}\}_{j < \omega}$ indiscernible in the theory of $\mathcal{M}_{p}^{\mathrm{Sk}}$. Let $\mathcal{M} = \mathrm{dcl}_{\mathrm{Sk}}(\bar{c}_{<\omega}) $, and let $\mathcal{M}'\succ \mathcal{M}$ be sufficiently saturated. Choose non-principal ultrafilters $U_{i}$, $i = 0, 1$, containing $c_{i,<\omega}$, and let the global types $p_{i}(x)=\{\varphi(x, c) \in \mathcal{M}': \varphi(\mathcal{M}, c) \in U_{i}\}$, so that each of the $p_{i}$ are finitely satisfiable over $\mathcal{M}$. It can be shown from $c_{0, j} \equiv^{\mathcal{L}_{p}^{\mathrm{Sk}}}_{c_{0, < j}c_{1, < j}} c_{1, j}$ that $p_{0}|_{\mathcal{M}} = p_{1}|_{\mathcal{M}}$; let $b$ realize this. Then we can choose $b$ so that a $p_{0}$-Morley sequence $\{b_{i}\}_{i < \omega}$ with $b_{0} = b$ is a clique, and a $p_{1}$-Morley sequence $\{b'_{i}\}_{i < \omega}$ with $b'_{0} = b$ is a $2$-anticlique.

Finally, we show, using the technique of Theorem 7.17 of \cite{Co15}, that the $R_{n}$ have the compatible order property ($\mathrm{SOP}_{3}$), Definition 3.10 of \cite{Mal10}. By compactness and Fact \ref{4-sop3fact}, the following will translate into an instance of $\mathrm{SOP}_{3}$ in $T$, a contradiction. We find an array $c_{0}, \ldots, c_{n}, \ldots, d_{0}, \ldots, d_{n}, \ldots$, with the following properties:

(1) For $m < n$, $d_{0}, \ldots, d_{m}, c_{m+1}, \ldots, c_{n}$ form a Morley sequence in $p_{0}$, so a clique.

(2) For $m < n$, $c_{m} d_{n} $ begin a Morley sequence in $p_{1}$, so $\neg R_{2}(c_{m}, d_{n})$. 

(3) $c_{0}, \ldots, c_{n}, \ldots\ind_{\mathcal{M}}^{K} d_{0}, \ldots, d_{n}, \ldots$

Suppose we have constructed $c_{0}, \ldots, c_{n}, d_{0}, \ldots, d_{n}$ satisfying these properties up to $n$. We find $c_{n+1}, d_{n+1}$. To find $d_{n+1}$, let $d'_{n+1} \models p_{0}(x)|_{\mathcal{M}d_{0}\ldots d_{n}}$ so $d'_{n+1} \ind^{K}_{\mathcal{M}}d_{0}, \ldots, d_{n}$. By (3), $c_{0}, \ldots, c_{n}\ind_{\mathcal{M}}^{K} d_{0}, \ldots, d_{n}$. Finally, if $d''_{n+1} \models p_{1}(x)|_{\mathcal{M}c_{0}\ldots c_{n}}$, by symmetry of Kim-independence, Fact \ref{4-symmetry}, $c_{0} \ldots c_{n} \ind^{K} d''_{n+1}$. So by the independence theorem (Fact \ref{4-independence}) and an automorphism, there is $d_{n+1} \models  p_{0}(x)|_{\mathcal{M}d_{0} \ldots d_{n}} \cup p_{1}(x)|_{\mathcal{M}c_{0}\ldots c_{n}}$ with $c_{0}, \ldots, c_{n} \ind_{\mathcal{M}}^{K} d_{0}, \ldots, d_{n} d_{n+1}$. Finally choose $c_{n+1} \models p_{0}(x)|_{\mathcal{M}c_{0}, \ldots c_{n} d_{0} \ldots d_{n}} $. It remains to show that this preserves (3). This follows from the following claim:

\begin{claim}
For any $a, b, c$, $M$, if $a \ind^{K}_{M} b$ and $\mathrm{tp}(c/Mab)$ extends to an $M$-invariant type $q(x)$, then $a \ind^{K}_{M} bc$.
\end{claim}

This follows from Claim \ref{4-inv5} below, using the fact that $\ind^{K^{+}} = \ind^{K}$ in the language of that claim (Kim's lemma, Fact \ref{4-kimslemma}, and compactness) and symmetry of $\ind^{K}$ (Fact \ref{4-symmetry}).

\end{proof}

\begin{example}\label{4-example3}
    Theorem \ref{3-main3} becomes false if we relax $\mathrm{NSOP}_{3}$ to $\mathrm{NSOP}_{4}$. Let $T$ be the model companion of (undirected) triangle-free tripartite graphs, with the partition denoted by unary predicates $P_{1}, P_{2}, P_{3}$. Let $M$ be a model and $p(x) = \{P_{1}(x) \vee P_{2}(x)\}$. Then $T$ is $\mathrm{NSOP}_{4}$, in fact a free amalgamation theory in the sense of Conant (\cite{Co15}). Moreover, $p(x)$ is internally $\mathrm{NSOP}_{1}$, in fact, internally simple. The associated theory has quantifier elimination in the language with unary predicates for $P_{1}$, $P_{2}$, $P_{m}$ denoting $x E m$ for each $m \in M$, and a binary relation symbol for the edge relation between elements of $P_{1}$ and elements of $P_{2}$. It is the model companion of graphs with interpretations for the unary predicates $P_{1}, P_{2}$ and $P_{m}$, such that $P_{1}$ and $P_{2}$ partition the graph and have no edges within them, there are no edges within $P_{m}$ for any $m \in M$, $P_{m_{1}}$ and $P_{m_{2}}$ are disjoint for $m_{1}, m_{2} \in M$ with $M \models m_{1} E m_{2}$, and for $i = 1, 2$ $P_{m}$ is disjoint from $P_{i}$ when $M \models P_{i} (m)$. In this form, the theory associated to $p(x)$ can be easily seen to be simple (for example, check that the relation $A \cap B = C$ coincides with forking-independence).
    
    However, $p(x)$ is not co-$\mathrm{NSOP}_{1}$: let $\varphi(x, y) =: xEy_{1} \wedge x E y_{2}$. For $\eta \in 2^{<\omega}$, choose $b_{\eta} = (b^{1}_{\eta}, b^{2}_{\eta})$ with, for any $\eta, \nu \in 2^{< \omega}$, $b^{i}_{\eta} \in P_{i}$, $\models \neg b^{i}_{\eta} E m$ for $i = 1, 2$ and $m \in M$, and $ \models b^{1}_{\eta} E b^{2}_{\nu}$ if and only if $\eta$ and $\nu$ are incomparable. This is possible, as we create no triangles. But $\varphi(x, y)$, $b_{\eta}$ witness the failure of Definition 3.1.1.
\end{example}

This proof can be viewed as an instance of a more general phenomenon. In this proof, the $R_{n}$ are the restriction to $p_{n}(\mathbb{M})$ of the \textit{characteristic sequence} of $\varphi(x, y)$, defined by Malliaris:

\begin{definition}(\cite{Mal10})
   Let $\varphi(x, y)$ be a formula. The \emph{characteristic sequence} of $\varphi(x, y)$ is the sequence of hypergraphs, on the vertices $\mathbb{M}^{|y|}$, defined by

   $$R_{n}(a_{1}, \ldots, a_{n}) =: (a_{1}, \ldots, a_{n}) \models \exists x \wedge_{i=1}^{n}\varphi(x, y_{i})$$
\end{definition}

On the other hand, within $\mathcal{M}_{p}$, the $R_{n}$ are just a sequence of hypergraphs, and do not describe a pattern of consistency internally to $\mathcal{M}_{p}$. Nonetheless, by showing that a particular configuration, the \textit{compatible order property}, arises among the $R_{n}$, we get a description of the complexity of $\varphi(x,y)$ \textit{in the original theory $T$}. In \cite{Mal10}, Malliaris introduces some hypergraph configurations corresponding, via the characteristic sequence, to consistency patterns (in the sense of \cite{GM22}) in a first-order formula. We introduce some additional definitions to cover the case of the tree property, $\mathrm{DCTP}_{2}$, and $\mathrm{SOP}_{1}$; the first of these comes from Observation 5.20 of \cite{Mal10}.

\begin{definition}
    Let $R_{\infty}=(V,\{R_{n}\}_{n < \omega})$ be a sequence of hypergraphs on a common set of vertices $V$, where $R_{n}$ is an $n$-ary edge relation Then $R_{\infty}$ has

    (1) An $(\omega, \omega, 1)$-array if there is an array $\{b_{ij}\}_{i, j \in \omega}$ such that there is some fixed $k$ such that, for all $i$, $\{b_{ij}\}_{j \in \omega}$ is a $k$-anticlique, and for any $\sigma \in \omega^{\omega}$, $\{ b_{i\sigma(i)}\}_{i \in \omega}$ is a clique. (Definition 3.4.2, \cite{Mal10}. By Claim 3.8, \cite{Mal10}, $\mathrm{TP}_{2}$ is equivalent to the presence of an $(\omega, \omega, 1)$-array in the characteristic sequence of a formula. 

    (2) The \emph{compatible order property} if there are $c_{0}, \ldots, c_{n}, \ldots, d_{0}, \ldots, d_{n}, \ldots$ such that for $m < n$, $d_{0}, \ldots, d_{m}, c_{m+1}, \ldots, c_{n}$ form a clique, while for $m < n$, $\neg R_{2}(c_{m}, d_{n})$. (Definition 3.10, \cite{Mal10}. In Conclusion 6.15 of \cite{Mal10}, $\mathrm{SOP}_{3}$ is shown to be equivalent to the compatible order property in the characteristic sequence of a formula.)

    (3) $\mathrm{MTP}$ if there is some fixed $k$ and parameters $\{b_{\eta}\}_{\eta \in \omega^{< \omega}}$ such that for each path $\sigma \in \omega^{\omega}$, $\{b_{\sigma|_{n} }\}_{n \in \omega}$ is a clique, but for each node $\eta \in \omega^{< \omega}$, $\{ b_{\eta \smallfrown \langle n \rangle }\}_{n \in \omega}$ is a $k$-anticlique. (In Observation 5.10 of \cite{Mal10}, the failure of a formula to be simple is observed to be equivalent to $\mathrm{MTP}$ in the characteristic sequence.)

    (4) $\mathrm{MSOP}_{1}$ if there are parameters $\{b_{\eta}\}_{\eta \in 2^{<\omega}}$ such that $\{ b_{\sigma \upharpoonleft n}\}_{n \in \omega}$ is a clique for any $\sigma \in 2^{\omega}$, but for any $\eta_{2} \unrhd \eta_{1} \smallfrown \langle 0\rangle$, $\{b_{\eta_{2}},  b_{\eta_{1} \smallfrown \langle 1\rangle}\}$ is a $2$-anticlique.

    (5) $\mathrm{MDCTP}_{2}$ if for some fixed $k$, there are parameters $\{b_{\eta}\}_{\eta \in 2^{<\omega}}$ such that $\{b_{\sigma \upharpoonleft n}\}_{n \in \omega}$ is a $k$-anticlique for any $\sigma \in 2^{\omega}$, but for any descending comb $\eta_{1} \ldots, \eta_{l} \in 2^{< \omega}$ (Definition \ref{4-descendingcomb}), $\{ b_{\eta_{i}}\}_{i=1}^{l}$ is a clique.
\end{definition}

\begin{remark}
    The letter $\mathrm{M}$ in $\mathrm{MTP}$, $\mathrm{MSOP}_{1}$ and $\mathrm{MDCTP}_{2}$ stands for \textit{Malliaris}. \end{remark}

Note that these properties are all \textit{graph-theoretic} in the sense of Malliaris, \cite{Mal10}, referring only to incidence patterns of the edges, rather than their consistency. They are similar in this sense to stability or $\mathrm{NIP}$, which make no reference to consistency but only ask for graph-theoretic configurations. In \cite{Sh90}, Shelah shows the following:

\begin{fact}
    Let $R(x, y)$ be an unstable formula, and assume that all Boolean combinations of instances of $R(x, y)$ are $\mathrm{NSOP}$. Then $R(x, y)$ has the independence property.
\end{fact}

Note that the form of this result is as follows: if a graph has one graph-theoretic configuration (instability), and an ambient model-theoretic tameness property ($\mathrm{NSOP}$, indeed quantifier-free $\mathrm{NSOP}$), then it has a more complicated graph-theoretic configuration (the independence property). It was Malliaris who first implicitly asked, in the context of strengthenings of the compatible order property (Remark 7.12, \cite{Mal12}), whether ambient classification-theoretic properties imply further graph-theoretic complexity gaps for hypergraphs. In the remainder of this section, we observe that model-theoretic tameness properties of hypergraph sequences that refer to consistency, such as simplicity and $\mathrm{NSOP}_{1}$, provide additional information about their graph-theoretic structure, just as Shelah shows a gap between simplicity and independence in $\mathrm{NSOP}$ graphs. We then further observe that the connection between internal properties of types and external properties of theories, including the aforementioned work of Chernikov on co-simplicity (\cite{Che14}), can be reinterpreted in terms of these graph-theoretic complexity gaps for model-theoretically tame hypergraphs.

\begin{proposition}\label{4-hyp}
    
Let $R_{\infty}=(V,\{R_{n}\}_{n < \omega})$ be sequence of hypergraphs on a common set of vertices $V$, where $R_{n}$ is an $n$-ary edge relation.

    (1) If $R_{\infty}$ is simple (in the hypergraph language) and has $\mathrm{MTP}$, it has an $(\omega, \omega, 1)$-array.
    
    (2) If $R_{\infty}$ is $\mathrm{NSOP}_{1}$ and has $\mathrm{MSOP}_{1}$, then it has $\mathrm{MDCTP}_{2}$ and the compatible order property. 

In fact, for (1), it suffices that no quantifier-free formula has the tree property, and similarly for (2) and $\mathrm{SOP}_{1}$.
\end{proposition}

\begin{proof} (Sketch)
    The argument for (1) is extracted from Chernikov's proof in \cite{Che14} that simple types are co-simple. In particular, we notice that the proof Lemma 6.13 of \cite{Che14} works when the rows are general indiscernible sequences, not just Morley sequences, and relies only on the internal simplicity of a type, not the full definition of a simple type. Suppose $R_{\infty}$ has $\mathrm{MTP}$, but is simple as a structure in the language of hypergraph sequences. By the proof, which can be found in a standard reference on simplicity theory such as \cite{Kim14}, that formulas with the tree property fail Kim's lemma for dividing, there is a model $M$ of the theory of $R_{\infty}$, some element $b$ of the monster model, and some indiscernible $k$-anticlique $I=\{b_{i}\}_{i \in \omega}$ starting with $b_{0}=b$, as well as a Morley sequence $J=\{b'_{i}\}_{i \in \omega}$ starting with $b'_{0}=b$ and forming a clique. Now suppose, by induction, that for $i \leq n$ there are are $I^{i}=\{b_{j}^{i}\}_{j < \omega}$ with $I^{i} \equiv_{M} I$ and $b_{0}^{i}=b'_{i}$ (so the $I^{i}$ are anticliques), such that for $\sigma \in \omega^{n}$, $\{b^{i}_{\sigma(i)}\}_{i \leq n} \smallfrown \{b'_{i}\}_{i \geq n+1}$ is a clique, and such that $I_{\leq n}\ind_{M} b'_{\geq n+1}$. By properties of independence in simple theories, $I_{\leq n} b_{>n+1}\ind_{M} b'_{ n+1}$. By the chain condition, take $I^{n+1} = \{b_{j}^{n+1}\}_{j < \omega}$, $I^{n+1} \equiv_{M} I$, with $b^{n+1}_{0}=b'_{n+1}$ such that $I_{n+1}$ is indiscernible over $MI_{\leq n} b'_{>n+1}$, and with $I_{\leq n} b'_{>n+1}\ind_{M} I_{ n+1}$. This suffices for the induction. Now the existence of an $(\omega, \omega, 1)$-array follows.

    For (2), suppose $R_{\infty}$ is $\mathrm{NSOP}_{1}$ in the hypergraph language, but has $\mathrm{MSOP}_{1}$. Then as in the proof of Theorem \ref{3-main3}, there is a model $\mathcal{M}$ of the theory of $R_{\infty}$ and there are $M$-invariant Morley sequences $\{b_{i}\}_{i < \omega}$ in the $M$-invariant type $p_{0}$ and $\{b'_{i}\}_{i < \omega}$ in the $M$-invariant type $p_{1}$, such that $b_{0}= b'_{0}$, $\{b_{i}\}_{i < \omega}$ is a clique, and $\{b'_{i}\}_{i < \omega}$ is a $2$-anticlique. To show $\mathrm{MDCTP}_{1}$, it suffices to find a tree $\{b_{\eta}\}_{\eta \in 2^{<\omega}}$ such that the paths, read downward, are Morley sequences in $p_{1}$, and the descending combs are Morley sequences in $p_{0}$. Formally, the construction will follow \cite{NSOP2}, Lemma 4.5. Say that a tree $\{c_{\eta}\}_{\eta \in 2^{\leq n}}$ is a \textit{generic tree} if for $\eta \in 2^{< n}$ $c_{\unrhd \eta \smallfrown \langle 0 \rangle } \ind^{K}_{M} c_{\unrhd \eta \smallfrown \langle 1 \rangle }$ (two subtrees at a node are Kim-independent), and $c_{\eta } \models p_{1}(x)|_{M c_{\rhd \eta}}$ (each node satisfies the restriction of $p_{1} (x)$ to its subtrees.) We prove the following claim (corresponding to Claim 4.6 of \cite{NSOP2}):

    \begin{claim}\label{4-gtree}
        Let $\{c_{\eta}\}_{\eta \in 2^{\leq n}}$ be a generic tree, and $A$ any set. Then there is some $\{c'_{\eta}\}_{\eta \in 2^{\leq n}} \equiv_{M} \{c_{\eta}\}_{\eta \in 2^{\leq n}}$ with $c'_{\eta} \models p_{0}(x)|_{MA}$ for each $\eta \in 2^{\leq n}$, and with $A \ind_{M}^{K} \{c'_{\eta}\}_{\eta \in 2^{\leq n}}$.
    \end{claim}

    \begin{proof}
        By induction on $n$, we may assume this is true for $\{c_{\eta}\}_{\eta \unrhd \langle 0 \rangle}$ and $\{c_{\eta}\}_{\eta \unrhd \langle 1 \rangle}$. Namely, find $\{c''_{\eta}\}_{\eta \unrhd \langle 0 \rangle} \equiv_{M}\{c_{\eta}\}_{\eta \unrhd \langle 0 \rangle}$  with $c''_{\eta} \models p_{0}(x)|_{MA}$ for each $\langle 0 \rangle \unlhd \eta \in 2^{\leq n}$ and $A \ind^{K}_{M} \{c''_{\eta}\}_{\eta \unrhd \langle 0 \rangle}$, and similarly, $\{c''_{\eta}\}_{\eta \unrhd \langle 1 \rangle} \equiv_{M}\{c_{\eta}\}_{\eta \unrhd \langle 1 \rangle}$  with $c''_{\eta} \models p_{1}(x)|_{MA}$ for each $\langle 1 \rangle \unlhd \eta \in 2^{\leq n}$ and $A \ind^{K}_{M} \{c''_{\eta}\}_{\eta \unrhd \langle 1 \rangle}$. Recall that $c_{\unrhd  \langle 0 \rangle } \ind^{K}_{M} c_{\unrhd \langle 1 \rangle }$ as $\{c_{\eta}\}_{\eta \in 2^{\leq n}}$ is a generic tree, so by the independence theorem and an automorphism, we can find $\{c'_{\eta}\}_{\eta \rhd  \langle \rangle } \equiv_{M}\{c_{\eta}\}_{\eta \rhd \langle \rangle}$ with $c'_{\eta} \models p_{0}(x)|_{MA}$ for each $\langle  \rangle \lhd \eta \in 2^{\leq n}$ and $A \ind^{K}_{M} \{c'_{\eta}\}_{\eta \rhd \langle  \rangle}$. Finally, by the independence theorem and an automorphism, find $c'_{\langle \rangle} \models p_{0} (x)|_{MA} \cup p_{1}(x)|_{M \{c'_{\eta}\}_{\eta \rhd  \langle \rangle }}$ such that $A \ind_{M}^{K} \{c'_{\eta}\}_{\eta \in 2^{\leq n}}$, as desired.
    \end{proof}

By induction, we construct a generic tree $\{b_{\eta}\}_{\eta \in 2^{<\omega}}$ such that the paths, read downward, are Morley sequences in $p_{1}$, and the descending combs are Morley sequences in $p_{0}$. Suppose we have constructed $I=\{b_{\eta}\}_{\eta \in 2^{\leq n}}$ with these properties. By Claim \ref{4-gtree}, we can find $I^{1} \equiv_{M} I \equiv_{M} I^{2}$ with $I^{1} \ind^{K}_{M} I^{2}$ and for $I^{1}=\{b^{1}_{\eta}\}_{\eta \in 2^{\leq n}}$, $I^{2}=\{b^{2}_{\eta}\}_{\eta \in 2^{\leq n}}$, $b^{2}_{\eta} \models p_{0}(x)|_{MI_{1}}$ for $\eta \in 2^{\leq n}$. The trees $I_{1}$ and $I_{2}$ of height $n$ will be the subtrees of our new generic tree of height $n+1$. Finally, let $b_{*} \models q_{1}(x)|_{MI_{1}I_{2}}$ be the new root. Reindexing accordingly, we get a generic tree $\{b_{\eta}\}_{\eta \in 2^{\leq n+1}}$ such that the paths, read downward, are Morley sequences in $p_{1}$, and the descending combs are Morley sequences in $p_{0}$. This completes the induction.

Now the compatible order property comes from the proof of Theorem \ref{3-main3}.

\end{proof}

\begin{example}
    If $R_{\infty}$ is the model companion of the empty theory in the language of hypergraph sequences (or, say, the theory axiomatizing the basic properties of characteristic sequences; see \cite{Mal10}, Observation 2.4), then it is a simple structure. But it has $\mathrm{MDCTP}_{2}$ and the compatible order property.
\end{example} 

We connect Proposition \ref{4-hyp} to the internal properties of types. We recall the definition of co-simplicity from \cite{Che14}:

\begin{definition}
    A type $p(x)$ over $A$ is co-simple if there is no formula $\varphi(x, y)$, $k \geq 2$ and parameters $\{b_{\eta}\}_{\eta \in \omega^{< \omega}}$, $b_{\eta} \subseteq p(\mathbb{M})$ such that for each path $\sigma \in \omega^{\omega}$, $\{\varphi(x, b_{\sigma|_{n} })\}_{n \in \omega}$ is consistent, but for each node $\eta \in \omega^{< \omega}$, $\{\varphi(x, b_{\eta \smallfrown \langle n \rangle })\}_{n \in \omega}$ is $k$-inconsistent.
\end{definition}

\begin{corollary}
(1) (\cite{Che14}, Theorem 6.17) In a $\mathrm{NTP}_{2}$ theory, internally simple types are co-simple. 

(2) In an $\mathrm{NDCTP}_{2}$ theory or an $\mathrm{NSOP}_{3}$ theory, internally $\mathrm{NSOP}_{1}$ types are co-$\mathrm{NSOP}_{1}$
\end{corollary}

\begin{proof}
    
    (1). If $p(x)$ is not co-simple then the restriction $R_{\infty}$ of some characteristic sequence to $p(\mathbb{M})$ has $\mathrm{MTP}$. If $p(x)$ is internally $\mathrm{NSOP}_{1}$, then $R_{\infty}$ is $simple$, so by the previous proposition it has an $(\infty, \infty, 1)$-array. So $T$ must have $\mathrm{TP}_{2}$.

    (2). If $p(x)$ is not co-$\mathrm{NSOP}_{1}$ then the restriction $R_{\infty}$ of some characteristic sequence to $p(\mathbb{M})$ has $\mathrm{MSOP}_{1}$. If $p(x)$ is internally $\mathrm{NSOP}_{1}$, then $R_{\infty}$ is $\mathrm{NSOP}_{1}$, so by the previous proposition it has $\mathrm{MDCTP}_{2}$ and the compatible order property. So $T$ must have $\mathrm{DCTP}_{2}$ and $\mathrm{SOP}_{3}$.
\end{proof}

In other words, the fact that internally $\mathrm{NSOP}_{1}$ types are co-$\mathrm{NSOP}_{1}$ in $\mathrm{NSOP}_{3}$ theories can be interpreted as saying that in an $\mathrm{NSOP}_{3}$ theory, the graph-theoretic complexity of a characteristic sequence must be reflected in its model-theoretic complexity in the hypergraph language.

\begin{remark}
    For $R_{\infty}$ a hypergraph sequence, define $R^{(m)}_{\infty}$ to be the hypergraph sequence whose vertices are $m$-tuples of vertices of $R_{\infty}$, and define $R_{n}^{(m)}((a^{1}_{1}, \ldots a^{m}_{1}), \ldots (a^{1}_{n}, \ldots a^{m}_{m})) =: R_{mn}(a^{1}_{1}, \ldots a^{m}_{1}, \ldots a^{1}_{n}, \ldots a^{m}_{m})$. For example, if $R_{\infty}$ is the characteristic sequence of $\varphi(x, y)$, then $R^{(m)}_{\infty}$ is the characteristic sequence of $\wedge_{i=1}^{m} \varphi(x, y_{i})$. If we consider hypergraphs up to the concatenation operation $R_{\infty} \mapsto R^{(m)}_{\infty}$, then we obtain additional information. For example, we can define $\mathrm{MSOP}_{2}=\mathrm{MTP}_{1}$ to mean that there exists a binary (or infinitely branching; see \cite{Adl07}, recounted in Fact 4.2 of \cite{CR15}) tree whose paths are cliques and whose incomparable pairs are $2$-anticliques. It follows from the proof of \cite{Sh90}, III.7.7, III.7.11 that up to concatenation, a hypergraph sequence with $\mathrm{MTP}$ either has an $(\omega, \omega, 1)$-array or has $\mathrm{MTP}_{1}$. So by Proposition \ref{4-hyp}.2, if $R_{\infty}$ is $\mathrm{NSOP}_{1}$, and has $\mathrm{MTP}$, then up to concatenation it either has $\mathrm{MTP}$ or an $(\omega, \omega, 1)$-array.

    It is also worth noting that if we define $\mathrm{MATP}$ to be the existence of a tree such that the antichains are cliques and the paths are $k$-anticliques, it follows from the proof of Theorem 4.8 of \cite{AK20} that up to concatenation, an $\mathrm{MSOP}_{1}$ hypergraph sequence has either $\mathrm{MSOP}_{2}$ or $\mathrm{MATP}$.

    Although it is shown in \cite{NSOP2} that $\mathrm{NSOP}_{1}$ coincides with $\mathrm{NSOP}_{2}$ for theories, the following informal question (which was asked by Byunghan Kim at the 2023 BIRS meeting on neostability theory) remains of interest: if a formula has $\mathrm{SOP}_{1}$, can it be shown that a related formula has $\mathrm{SOP}_{2}$?   There are $\mathrm{SOP}_{1}$ formulas $\varphi(x, y)$ such that $\wedge_{i =1}^{n} \varphi(x, y_{i})$ is $\mathrm{NSOP}_{2}$; see \cite{AK20}, §6. But we could ask, say, whether if a formula has $\mathrm{SOP}_{1}$, $\mathrm{SOP}_{2}$ must appear in the quantifier-free formulas of its characteristic sequence. At the graph-theoretic level, we may try to show that if we assume only quantifier-free formulas of a hypergraph sequence $R_{\infty}$ to be $\mathrm{NSOP}_{2}$, then if $R_{\infty}$ has $\mathrm{MSOP}_{1}$, it has $\mathrm{MSOP}_{2}$ up to concatenation. The following argument, applying the arguments in arguments in \cite{NSOP2}, appears initially to work, but fails at an important step.

    If $R_{\infty}$ has $\mathrm{MSOP}_{1}$, there is a model $M$ and two $M$-finitely satisfiable Morley sequences, one of which is a clique and one of which is a $2$-anticlique. (That $k$ may chosen to be $2$ comes from \cite{Lee22}, \cite{KR17}). Now Lemma 4.5 of \cite{NSOP2} says that for any coheir $p(x)$ over $M$ and canonical coheir $q(x)$ over $M$, there is a tree whose paths are Morley sequences in $p(x)$ and whose descending combs are Morley sequences in $q(x)$. So if Morley sequences in $q(x)$ are anticliques and Morley sequences in $p(x)$ are cliques, the descending combs will be anticliques and the paths will be cliques. By the proof of Lemma 2.8 of \cite{NSOP2} ($\mathrm{SOP}_{2}$ = k-$\mathrm{DCTP}_{1}$), such a tree, up to concatenation, gives an instance of $\mathrm{MSOP}_{2}$.

    So because there exists a finitely satisfiable Morley sequence that is a clique, either $R_{\infty}$ has $\mathrm{MSOP}_{2}$ and we are done, or there is also a \textit{canonical} Morley sequence that is a clique. At this point, now that we have an $M$-finitely satisfiable Morley sequence that is a $2$-anticlique and a canonical Morley sequence that is a clique, we can prove Kim's lemma for canonical Morley sequences, symmetry for Conant-independence, and the weak independence theorem for Conant-independence, all in the quantifier-free context as for $\mathrm{NSOP}_{2}$ theories.

    The argument is valid up to this point. But then, we may try to use the technique of Conant (\cite{Co15}) developed in Section 6 of \cite{NSOP2} to show the compatible order property; up to concatenation, the compatible order property would imply $\mathrm{MSOP}_{2}$ (\cite{Mal10}, Observation 3.11). This strategy fails for this reason: in the notation of Section 6 of \cite{NSOP2}, after fixing the strong canonical coheir $q(x)$, we want to preserve at each stage the property (property 1) that the sequence beginning with $b_{i}^{2}$ for $i < n$, and then continuing with $b_{i}^{1}$ for $i \geq n$, is a strong canonical Morley sequence in $q(x)$. However, since the weak independence theorem for Conant-independence only applies in the quantifier-free context in $R_{\infty}$, we do not actually preserve this property when applying the weak independence theorem.

\end{remark}

\section{Independence of internally $\mathrm{NSOP}_{1}$ types in $\mathrm{NSOP}_{3}$ theories}

In this section, we prove an extension of the independence theorem of Kaplan and Ramsey (\cite{KR17}) to internally $\mathrm{NSOP}_{1}$ types in $\mathrm{NSOP}_{3}$ theories. We will use Theorem \ref{3-main3}, namely that internally $\mathrm{NSOP}_{1}$ types in $\mathrm{NSOP}_{3}$ theories are co-$\mathrm{NSOP}_{1}$.

While the theorem does not give $a_{1}'' \ind_{M}^{K^{*}} a_{2}a_{3}$, $a''_{1}$ can be chosen so that any two of $a''_{1}, a_{2}, a_{3}$ is Conant-independent from the third, somewhat similarly to Theorem 2.13 of \cite{KR18}. 

\begin{theorem} \label{4-main4}
    Let $T$ be $\mathrm{NSOP}_{3}$, and let $p_{1}, p_{2}, p_{3}$ be internally $\mathrm{NSOP}_{1}$ types over $M$. Let $a_{1} \equiv_{M} a'_{1} \subset p_{1}(\mathbb{M})$, $ a_{2} \subset p_{2}(\mathbb{M})$,  $ a_{3} \subset p_{3}(\mathbb{M})$.  If $a_{1} \ind^{K^{*}}_{M} a_{2}$, $a'_{1} \ind^{K^{*}}_{M} a_{3}$, $a_{2} \ind^{K^{*}}_{M} a_{3}$, there is some $a''_{1}$ with $a''_{1} \models \mathrm{tp}(a_{1}/Ma_{2}) \cup \mathrm{tp}(a'_{1}/Ma_{3})$. Moreover, $a''_{1}$ can be chosen with $a_{2}a_{3} \ind^{K^{*}}_{M}a''_{1}$, $a_{2}a''_{1} \ind^{K^{*}}_{M} a_{3}$ and $a_{3}a''_{1} \ind^{K^{*}}_{M} a_{2}$.
\end{theorem}

\begin{proof}
    
We begin with some observations on co-$\mathrm{NSOP}_{1}$ types. First of all, Conant-independence between co-$\mathrm{NSOP}_{1}$ types is just Kim-dividing independence.

\begin{claim}\label{4-kd4}
If  $b \subset p(\mathbb{M})$ for $p(x)$ a co-$\mathrm{NSOP}_{1}$ type over $M$, then $a \ind^{K^{*}}_{M} b $ if and only if $a \ind^{Kd}_{M} b $.   
\end{claim}

\begin{proof}
   If $a \ind^{K^{*}}_{M} b $, then $a \ind^{Kd}_{M} b $ by Kim's lemma, Definition 3.1.3. Conversely, if $a \ind^{Kd}_{M} b $, then in particular, by compactness we can choose an $M$-finitely satisfiable Morley sequence $\{b_{i}\}_{i < \omega}$, $b_{0}=b$, that is indiscernible over $Ma$. But then, by Fact \ref{4-conantforkdivide} below, formulas that do not Kim-divide over $M$ by a some $M$-finitely satisfiable Morley sequence do not Conant-fork over $M$, so $a \ind^{K^{*}}_{M} b $. 
\end{proof}

\begin{claim}\label{4-symm4}
If $p(x), q(x)$ are co-$\mathrm{NSOP}_{1}$ types over $M$ and $a \subseteq p(\mathbb{M})$, $b \subseteq q(\mathbb{M})$, then $a\ind^{K^{*}}_{M}b$ if and only if $b\ind^{K^{*}}_{M}a$.    
\end{claim}

This follows from Proposition 3.1.

\begin{claim}\label{4-cc4}
If  $b \subset p(\mathbb{M})$ for $p(x)$ a co-$\mathrm{NSOP}_{1}$ type over $M$, and $a \ind^{K^{*}}_{M} b $, then for any $M$-finitely satisfiable type $q(x)$, there is a Morley sequence $I=\{b_{i}\}_{i < \omega}$ in $q(x)$, $b_{0} =b$, that is indiscernible over $Ma$, and any such Morley sequence will satisfy $a \ind^{K^{*}}_{M} I$. 
\end{claim}

This is the ``chain condition", Claim \ref{4-cc3}, together with compactness. We use Claims \ref{4-kd4}-\ref{4-cc4} throughout.

Third, we have the weak independence theorem between \textit{two} co-$\mathrm{NSOP}_{1}$ types, analogously to Proposition 6.1 of Kaplan and Ramsey, \cite{KR17}:

\begin{claim}\label{4-wit4}
 Let $p(x), q(x)$ be co-$\mathrm{NSOP}_{1}$ types over $M$, and let $a \equiv_{M} a' \subseteq p(\mathbb{M})$, $b, c \subseteq q(\mathbb{M})$, $a \ind^{K^{*}}_{M} b$, $a' \ind^{K^{*}}_{M} c$, $c \ind^{u}_{M} b$. Then there is $a'' \ind_{M}^{K^{*}} bc$  with $a'' \models \mathrm{tp}(a/Mb) \cup \mathrm{tp}(a'/Mc)$. 
\end{claim}

\begin{proof}
    The proof is similar to Proposition 6.1 of \cite{KR17}. By Claims \ref{4-symm4} and \ref{4-cc4} and $a \ind^{K^{*}}_{M} b$, let $I = \{a_{i}\}_{i < \omega}$ be an $M$-invariant Morley sequence with $a_{0} =a$ that is indiscernible over $Mb$. Again by Claims \ref{4-symm4} and \ref{4-cc4}, $a' \ind^{K^{*}}_{M} c$ and $a' \equiv_{M} a$, for $r(x, y) = \mathrm{tp}(a' c/M)$, $\cup_{i < \omega} r(a_{i}, y)$ is consistent, so we can choose some $c' \models \cup_{i < \omega} r(a_{i}, y) $. By Ramsey's theorem, compactness and an automorphism, we can choose $c'$ in particular so that $I$ remains indiscernible over $Mbc'$. So $bc' \ind_{M}^{K^{*}} a$, and $a \ind_{M}^{K^{*}} bc'$, with $c'a \equiv_{M} ca'$.

Let $s(y)$ be an $M$-finitely satisfiable type extending $\mathrm{tp}(c/Mb)$, and let $c'' \models s(y)|_{Mbc'}$, so $c'' \equiv_{Mb} c$ and $c''\ind_{M}^{u} bc'$. As $c'' \equiv_{M} c$, choose $b'$ with $b'c'' \equiv_{M} bc'$; by left extension for $\ind_{M}^{u}$, $b'$ can further be chosen with $b'c''\ind_{M}^{u} bc'$. Then $bc', b'c''$ begin an $M$-invariant Morley sequence $J$. As $a \ind_{M}^{K^{*}} bc'$, there is an $M$-invariant Morley sequence $J' \equiv_{Mbc'} J$ indiscernible over $Ma$; using Claim \ref{4-cc4}, $a \ind_{M}^{K^{*}} J'$. Write $J'=(bc', b'''c''', \ldots)$. Then $c''' a\equiv_{M} c'a \equiv_{M} ca'$, $c''' \equiv_{Mb} c'' \equiv_{Mb} c$ and $a \ind_{M}^{K^{*}}bc'''$. By an $Mb$-automorphism taking $c'''$ to $c$, we obtain $a'$ as desired.

\end{proof}

Using this weak independence theorem, we can now show the full independence theorem between \textit{two} co-$\mathrm{NSOP}_{1}$ types.

\begin{claim}\label{4-ind4}
 Let $p(x), q(x)$ be co-$\mathrm{NSOP}_{1}$ types over $M$, and let $a \equiv_{M} a' \subseteq p(\mathbb{M})$, $b, c \subseteq q(\mathbb{M})$, $a \ind^{K^{*}}_{M} b$, $a' \ind^{K^{*}}_{M} c$, $c \ind^{K^{*}}_{M} b$. Then there is $a'' \ind_{M}^{K^{*}} bc$  with $a'' \models \mathrm{tp}(a/Mb) \cup \mathrm{tp}(a'/Mc)$. 
\end{claim}

\begin{proof}
We could have followed the proof of Theorem 6.5 of \cite{KR17}, but we offer our own exposition. We first show that it suffices to show consistency of $\mathrm{tp}(a/Mb) \cup \mathrm{tp}(a'/Mc)$ in the statement of Claim \ref{4-theclaim4}. Let $r(x)$ be an $M$-finitely satisfiable type extending $\mathrm{tp}(a/M)$. By Claims \ref{4-symm4} and \ref{4-cc4}, we can find $I= \{a_{i}\}_{i < \omega}$ with $a_{0} =a$ indiscernible over $Mb$ and $I'= \{a'_{i}\}_{i < \omega}$ with $a'_{0} =a'$ indiscernible over $Mc$, both $M$-finitely satisfiable Morley sequences in $r(x)$. Then $I \equiv_{M} I'$. So using the consistency, we can find $I'' \models (I/Mb) \cup \mathrm{tp}(I'/Mc)$, which by Ramsey's theorem and compactness, can be assumed indiscernible over $Mbc$. So $I'' \ind^{K}_{M} bc$ by Claim \ref{4-symm4}, and we can find $a''$ in this sequence as desired.

   Therefore, suppose $a, a', b, c$ are as in the statement of the claim, and $\mathrm{tp}(a/Mb) \cup \mathrm{tp}(a'/Mc)$ is inconsistent. By compactness, there are some $\varphi(x, b) \in \mathrm{tp}(a/Mb)$ and  $\psi(x, c) \in \mathrm{tp}(a'/Mc)$ with $\{\varphi(x, b) , \psi(x, c)\} $ inconsistent. Let $s(y, z) = \mathrm{tp}(bc/M)$. We find $b_{1}, \ldots, b_{n} \ldots \subset q(\mathbb{M})$, $c_{1}, \ldots, c_{n} \subset q(\mathbb{M})$ with the following properties:

   (1) For $m < n$, $b_{i} \ind_{M}^{u} b_{i-1} \ldots b_{1}$ for $i \leq m$, and $c_{i} \ind_{M}^{u} c_{i-1} \ldots c_{m+1}b_{m} \ldots b_{1}$ for $m \leq i \neq n$. Thus by repeated applications of Claim \ref{4-wit4}, $\{\varphi(x, b_{1}), \ldots \varphi(x, b_{m}), \psi(x, c_{m+1}), \ldots, \psi(x,c_{n})\}$ is consistent.

   (2) For $i < j$, $b_{i}c_{j} \models s(y, z)$, so $\{\varphi(x, b_{i}), \psi(y, c_{j})\}$ is inconsistent.

   (3) $b_{1} \ldots, b_{n} \ldots \ind_{M}^{K^{*}} c_{1} \ldots, c_{n} \ldots$ .

By Fact \ref{4-sop3fact}, this will give us a failure of co-$\mathrm{NSOP}_{1}$, a contradiction. We use the technique of Conant, \cite{Co15} (though it is not yet necessary to get $\mathrm{SOP}_{3}$; this technique is similar to the ``zig-zag lemma," Lemma 6.4, from the original proof of the independence theorem in \cite{KR17}). Assume $b_{1}, \ldots, b_{n}, c_{1}, \ldots, c_{n}$ already constructed, satisfying these properties up to $n$ (including $c_{n} \equiv_{M} c$). By repeated instances of Claim \ref{4-wit4} (applied to $q(x), q(x)$), and $b \ind^{K^{*}}_{M} c$, there is $b'_{n+1} \ind^{K^{*}} c_{1} \ldots c_{n}$ with $b'_{n+1} \models \cup_{i =1}^{n} s(y, c_{i})$. Again by Claim \ref{4-wit4}, $b_{n} \ldots b_{1} \ind^{K^{*}}_{M} c_{1}, \ldots c_{n}$, and an automorphism, we can additionally choose $b'_{n+1}=b_{n+1}$ so that $b_{n+1} \ind_{M}^{u} b_{1}\ldots b_{n} $ and $b_{n+1} \ldots b_{1} \ind^{K^{*}}_{M} c_{1}, \ldots c_{n}$. Now choose $c_{n+1} \equiv_{M} c$ with $c_{n+1} \ind^{u} c_{1} \ldots c_{n} b_{1} \ldots b_{n+1}$. We get $b_{n+1} \ldots b_{1} \ind^{K^{*}}_{M} c_{1}, \ldots c_{n+1}$, by the proof of Claim \ref{4-inv5}.
\end{proof}

We first show that the ``moreover" clause follows from the conclusion. Let $a_{1} \equiv_{M} a'_{1} \subseteq p_{1}(\mathbb{M})$, $ a_{2} \models p_{2}(\mathbb{M})$,  $ a_{3} \models p_{3}(\mathbb{M})$ be as in the hypotheses of the theorem. As $a_{2} \ind^{K^{*}}a_{3}$, by Claims \ref{4-kd4} and \ref{4-cc4}, there is an $M$-finitely satisfiable Morley sequence $I_{2}=\{a_{2, i}\}_{i < \omega}$ with $a_{2, 0}=a_{2}$ that is indiscernible over $Ma_{3}$ and $I_{2} \ind^{K^{*}}a_{3}$. Likewise, there is an $M$-finitely satisfiable Morley sequence $I_{3}=\{a_{3, i}\}_{i < \omega}$ with $a_{3, 0}=a_{3}$ that is indiscernible over $MI_{2}$ and with $I_{2} \ind^{K^{*}}I_{3}$. By $a_{1} \ind^{K^{*}}_{M} a_{2}$ and an automorphism, we can find $a^{*}_{1} \equiv_{Ma_{2}} a_{1}$ with $a^{*}_{1} \ind^{K^{*}}_{M} I_{2}$ and $I_{2}$ indiscernible over $Ma^{*}_{1}$, so we can assume $a_{1} \ind^{K^{*}}_{M} I_{2}$ and $I_{2}$ is indiscernible over $Ma_{1}$. Similarly, we can assume $a'_{1} \ind_{M}^{K^{*}}I_{3}$ and $I_{3}$ is indiscernible over $Ma'_{1}$. Fix an $M$-finitely satisfiable type $q(x)$ extending $\mathrm{tp}(a_{1}/M)$. Then there is a Morley sequence $I_{1} = \{a_{1, i}\}_{i < \omega}$ in $q(x)$ with $a_{1, 0}=a_{1}$, $I_{1} \ind^{K^{*}}_{M} I_{2}$ and $I_{1}$ indiscernible over $I_{2}$. There is also a Morley sequence $I'_{1} = \{a'_{1, i}\}_{i < \omega}$ in $q(x)$ with $a'_{1, 0}=a'_{1}$, $I'_{1} \ind^{K^{*}}_{M} I_{3}$ and $I'_{1}$ indiscernible over $I_{3}$. Since the Morley sequences were chosen to be in the same $M$-finitely satisfiable type, $I_{1} \equiv_{M} I'_{1}$. So applying the consistency part of the theorem, we can find $\{a''_{1, i}\}_{i < \omega}=I''_{1} \models \mathrm{tp}(I_{1}/MI_{2}) \cup \mathrm{tp}(I'_{1}/MI_{3})$. For $j, k, \ell < \omega$ $a''_{1, j}a_{2, k} \equiv_{M} a_{1}a_{2}$, $a''_{1, j}a_{3, \ell} \equiv_{M} a'_{1}a_{3}$, $a_{2, k}a_{3, \ell} \equiv_{M} a_{2}a_{3}$. We apply the following case of Lemma 1.2.1 \cite{Che14}:

\begin{fact}\label{4-indisc4}
Let $I''_{1}=\{a''_{1, i}\}_{i < \omega}$, $I_{2}=\{a_{2, i}\}_{i < \omega}$, $I_{3}=\{a_{3, i}\}_{i < \omega}$ be indiscernible sequences over $M$. Then there are mutually indiscernible $I'''_{1}=\{a'''_{1, i}\}_{i < \omega}$, $I'''_{2}=\{a'''_{2, i}\}_{i < \omega}$, $I'''_{3}=\{a'''_{3, i}\}_{i < \omega}$, (i.e. each $I'''_{m}$ indiscernible over $MI'''_{\neq m}$) such that for any formula $\varphi(\overline{x}, \overline{y}, \overline{z}) \in L(M)$, if for all $\overline{j}, \overline{k}, \overline{l}$ with $j_{1}< \ldots<  j_{n} < \omega$, $k_{1}< \ldots<  k_{n} < \omega$, $\ell_{1}< \ldots<  \ell_{n} < \omega$, $\models \varphi(\overline{a}''^{\overline{j}}_{1}, \overline{a}^{\overline{k}}_{2}, \overline{a}^{\overline{\ell}}_{3})  $, then for all such $\overline{j}, \overline{k}, \overline{l}$, $\models \varphi(\overline{a}'''_{1, \overline{j}}, \overline{a}'''_{2, \overline{k}}, \overline{a}'''_{3, \overline{l}})  $ 
\end{fact}

Let $I'''_{1}=\{a'''_{1,, i}\}_{i < \omega}$, $I'''_{2}=\{a'''_{2, i}\}_{i < \omega}$, $I'''_{3}=\{a'''_{3, i}\}_{i < \omega}$ be as in Fact 4.3. Then ${a'''_{2, 0}}{a'''_{3, 0}}\ind^{K^{*}}_{M}{a'''_{1, 0}}$, ${a'''_{1, 0}}{a'''_{2, 0}}\ind^{K^{*}}_{M}{a'''_{3, 0}}$, and ${a'''_{1, 0}}{a'''_{3, 0}}\ind^{K^{*}}_{M}{a'''_{2, 0}}$, and ${a'''_{1, 0}} {a'''_{2, 0}} \equiv_{M} a_{1}a_{2}$, ${a'''_{1, 0}} {a'''_{3, 0}} \equiv_{M} a'_{1}a_{3}$, ${a'''_{2, 0}} {a'''_{3, 0}} \equiv_{M} a_{2}a_{3}$. So by an automorphism, we find $a''_{1}$ as desired in the ``moreover" clause.

We finally show the actual consistency part of the theorem. Let $q_{1}(y, z) = \mathrm{tp}(a_{2} a_{3}/M)$, $q_{2}(x, z) = \mathrm{tp}(a'_{1} a_{3}/M)$, $q_{3}(x, y) = \mathrm{tp}(a_{1} a_{2}/M)$. By an automorphism, it sufficies to show that $q_{1}(y, z) \cup q_{2}(x, z) \cup q_{3}(x, y)$ is consistent. This will require $\mathrm{NSOP}_{3}$; formally, we will use the technique of Evans and Wong, from Theorem 2.8 of \cite{EW09}. Call $A \subset p_{1}(\mathbb{M})$, $B \subset p_{2}(\mathbb{M})$, $C \subset p_{3}(\mathbb{M})$ a \textit{generic triple} if there are mutually indiscernible $M$-invariant Morley sequences $I_{A} = \{A_{i}\}_{i < \omega}$ with $A_{0} = A$, $I_{B} = \{B_{i}\}_{i < \omega}$ with $B_{0} = B$,  $I_{C} = \{C_{i}\}_{i < \omega}$ with $C_{0} = C$; note that it follows that $A$, $B$ and $C$ are pairwise Conant-independent.

\begin{claim}\label{4-theclaim4}
    Let $A$, $B$, $C$ be a generic triple, and $b \subseteq p_{2}(\mathbb{M})$ such that $A \ind^{K^{*}}_{M} b$. Then there is some $b' \equiv_{MA} b$ with $b' \ind^{K^{*}}_{M} B$ such that $A, Bb', C$ form a generic triple.
\end{claim}

\begin{proof}
    Let $I_{A}$, $I_{B}$, $I_{C}$ be as in the definition of a generic triple.

    \begin{subclaim}
    There is $b'' \equiv_{M} b$ and $I_{b''}=\{b''_{i}\}_{i < \omega}$ with $b''_{0} = b''$ such that $\{B_{i} b''_{i}\}_{i < \omega}$ forms an invariant Morley sequence over $M$ and $I_{B} \ind^{K^{*}}_{M} I_{b''}$.
    
    \end{subclaim}

    \begin{proof}
    Let $I_{B}$ be a Morley sequence in the $M$-invariant type $p(X)$. Choose an $M$-invariant type $q(x)$ extending $\mathrm{tp}(b/M)$. By an automorphism, there is $I_{b''}=\{b''_{i}\}_{i < \omega}$ such that, for $n < \omega$, $b''_{n}B_{n} \ldots b''_{0}B_{0}\models (q(x)\otimes p(X))^{(n)} $. So $\{B_{i} b''_{i}\}_{i < \omega}$ is an $M$-invariant Morley sequence, and $b''_{0} \equiv_{M} b$; define $b'' := b''_{0}$. Finally, we show that $I_{B} \ind^{K^{*}}_{M} I_{b''}$. Suppose by induction that $B_{0} \ldots B_{n}\ind^{K^{*}}_{M}b''_{0} \ldots b''_{n}$. Note that $\mathrm{tp}(B_{n+1}/MB_{0}\ldots B_{n} b''_{0} \ldots b''_{n})$ extends to an $M$-invariant type, and $\mathrm{tp}(b''_{n+1}/MB_{0}\ldots B_{n}B_{n+1} b''_{0} \ldots b''_{n})$ also extends to a global $M$-invariant type. By the proof of Claim \ref{4-inv5} below, we see that for any sets $e, f, g$ of realizations of a common co-$\mathrm{NSOP}_{1}$ type over $M$, if $e \ind^{K^{*}}_{M} f$ and $\mathrm{tp}(g/Mef)$ extends to a global $M$-invariant type $q(x)$, then $e \ind^{K^{*}}_{M} fg$. So by two applications of this fact and symmetry (Claim \ref{4-symm4}), $B_{0} \ldots B_{n}B_{n+1}\ind^{K^{*}}_{M}b''_{0} \ldots b''_{n}b''_{n+1}$. This completes the induction, from which it follows that $I_{B} \ind^{K^{*}}_{M} I_{b''}$.
    \end{proof}

    Let $b''$ be as in the subclaim. As $A \ind^{K^{*}}_{M} b$, for $p(X, y) = \mathrm{tp}(Ab/ M)$, by Claim \ref{4-cc4} and an automorphism there is $A' \models \cup_{i < \omega} p(X, b''_{i})$ with $A' \ind^{K^{*}}_{M} I_{b''}$. By Claims \ref{4-symm4} and \ref{4-cc4}, we can then find $\{A'_{i}\}=I_{A'} \equiv_{M} I_{A}$ indiscernible over $MI_{b''}$ with $A'_{0}=A'$ and $I_{A'} \ind^{K^{*}}_{M} I_{b''}$. So we have $I_{A'} \equiv_{M} I_{A}$ and $I_{A'} \ind^{K^{*}}_{M} I_{b''}$, $I_{A} \ind^{K^{*}}_{M} I_{B}$ by indiscernibility of $I_{B}$ over $I_{A}$ and claim \ref{4-cc4}, and $I_{B} \ind^{K^{*}}_{M} I_{b''}$ by the subclaim. So by the independence theorem between two co-$\mathrm{NSOP}_{1}$ types (Claim \ref{4-ind4}) and an automorphism, there is some $\{b^{*}_{i}\}_{i < \omega}=I^{*}_{b''}$ with $ I^{*}_{b''} \equiv_{MI_{B}} I_{b''}$ and $I^{*}_{b''} I_{A} \equiv_{M} I_{b''} I_{A'}$. The sequence $\{b^{*}_{i}\}_{i < \omega}=I^{*}_{b''}$ will have the following three properties: $b^{*}_{i} \models p(A_{j}, y)$, so $b^{*}_{i} \equiv_{MA_{j}} b$, for $i, j < \omega$, $\{B_{i}b^{*}_{i}\}_{i < \omega}$ form an $M$-invariant Morley sequence, and $b^{*}_{i} \ind_{M}^{K^{*}} B_{i}$ for $i < \omega$. If we extract mutually indiscernible  sequences from $I_{A}, I_{B}I^{*}_{b''}, I_{C}$, finding $\check{I}_{A}, \check{I}_{B}\check{I}^{*}_{b''}, \check{I}_{C}$ as in Fact \ref{4-indisc4}, then $I_{A} I_{B} I_{C} \equiv_{M} \check{I}_{A} \check{I}_{B} \check{I}_{C}$, so we may assume $\check{I}_{A} = I_{A}$, $\check{I}_{B} = I_{B}$, $\check{I}_{C} = I_{C}$ and then $\check{I}^{*}_{b''}=\{\check{b}_{i}^{*}\}_{i < \omega}$ will also have these three properties that $\{b^{*}_{i}\}_{i < \omega}=I^{*}_{b''}$ has. Let $b' =\check{b}_{0}^{*}$. Then $b'  \equiv_{MA} b$, $b' \ind_{M}^{K} B$ and $I_{A}, I_{B}\check{I}^{*}_{b''}, I_{C}$ will be mutually indiscernible $M$-invariant Morley sequences, so $A, Bb', C$ form a generic triple.  
\end{proof}

Now we find $a^{1}, \ldots, a^{n}, \ldots \models \mathrm{tp}(a_{1}/M)$, $b^{1}, \ldots, b^{n}, \ldots \models \mathrm{tp}(a_{2}/M)$, and $c^{1}, \ldots, c^{n}, \ldots \models \mathrm{tp}(a_{3}/M)$ with the following properties:

(1) For $i < j$, $a^{j}c^{i} \models q_{2}(x, z)$, $a^{i}b^{j} \models q_{3}(x, y)$, $b^{i}c^{j} \models q_{1}(y, z)$.

(2) For $i < \omega$, $a^{i} \ind^{K}_{M} a^{1} \ldots a^{i-1}$, $b^{i} \ind^{K^{*}}_{M} b^{1} \ldots b^{i-1}$, and $c^{i} \ind^{K^{{*}}}_{M} c^{1} \ldots c^{i-1}$

(3) For each $n< \omega$ $a^{1} \ldots a^{n}$, $b^{1} \ldots b^{n}$, $c^{1} \ldots c^{n}$ form a generic triple.

Assume $a^{1}, \ldots, a^{n}$, $b^{1}, \ldots, b^{n}$, $c^{1}, \ldots, c^{n}$ already constructed, satisfying these properties up to $n$. As for $i \leq n$, $a^{i} \ind^{K^{*}}_{M} a^{1} \ldots a^{i-1}$, we can find some $b \models \cup^{n}_{i=1} q_{3}(a^{i}, y)$ with $b \ind_{M}^{K^{*}} a^{1} \ldots a^{n}$, by, say, repeated applications of the independence theorem between two co-$\mathrm{NSOP}_{1}$ types, Claim \ref{4-ind4} (though we could have stated the claim so that we need less than this). Then letting $a^{1}, \ldots, a^{n}=A$, $b^{1}, \ldots, b^{n}=B$, $c^{1}, \ldots, c^{n}=C$, we can choose $b_{n+1}=b'$ as in Claim \ref{4-theclaim4}, while will be as desired. Symmetrically, we find $c_{n+1}$ and $a_{n+1}$.

Now let 

$$\Phi(x^{1}, y^{1}, z^{1}; x^{2}, y^{2}, z^{2}) = q_{2}(x^{2}, z^{1}) \cup q_{3}(x^{1},y^{2}) \cup q_{1}(y^{1}, z^{2})$$ By (1) this has infinite chains, so by $\mathrm{NSOP}_{3}$ and compactness it has a $3$-cycle: some 

$$(d^{1}, e^{1}, f^{1}, d^{2}, e^{2}, f^{2}, d^{3}, e^{3}, f^{3})$$ $$\models \Phi(x^{1}, y^{1}, z^{1}; x^{2}, y^{2}, z^{2}) \cup \Phi(x^{2}, y^{2}, z^{2}; x^{3}, y^{3}, z^{3}) \cup \Phi(x^{3}, y^{3}, z^{3}; x^{1}, y^{1}, z^{1})$$ In particular, $(d^{1}, e^{2}, f^{3}) \models q_{3}(x,y) \cup q_{1}(y, z) \cup q_{2}(x, z)$, as desired. This concludes the proof of the theorem. \end{proof}

It is of interest that the conclusion of Theorem \ref{4-main4} does not hold for $\mathrm{NSOP}_{4}$ theories, nor does it follow from co-$\mathrm{NSOP}_{1}$.

\begin{example}
    Let $T$ be the model companion of the theory of triangle-free tripartite graphs, with the partition denoted by $P_{1}(x), P_{2}(x), P_{3}(x)$ as in Example \ref{4-example3}. Recall that $T$ is $\mathrm{NSOP}_{4}$, and $T$ is a free amalgamation theory in the sense of \cite{Co15}, so $a \ind_{M}^{K^{*}} b$ if and only of $a \cap b \subseteq M$; see Proposition 4.3 of \cite{GFA}. Let $p_{i}(x) := P_{i}(x)$ for $i = 1, 2, 3$. Then the $p_{i}(x)$ are internally stable--the structures $\mathcal{M}_{p_{i}}$ have quantifier elimination in the unary language of $M$-definable subsets of $P_{i}(x)$. Internally stable types are always co-$\mathrm{NSOP}_{1}$: by the proof of Theorem \ref{3-main3}, if an internally stable type $p$ is not co-$\mathrm{NSOP}_{1}$, then in the theory of $\mathcal{M}_{p}$, there is a hypergraph sequence $\{R_{n}\}$, a model $\mathcal{M}$, and Morley sequences $\{b_{i}\}_{i < \omega}$ and $\{b'_{i}\}_{i < \omega}$ with $b_{0} \equiv_{\mathcal{M}}b'_{0}$ such that $\{b_{i}\}_{i < \omega}$ is a clique and $\{b'_{i}\}_{i < \omega}$ is an anti-clique. But this is impossible if the theory of $\mathcal{M}_{p}$ is stable, as $b_{0} \equiv_{\mathcal{M}}b'_{0}$ implies $\{b_{i}\}_{i < \omega} \equiv_{M} \{b'_{i}\}_{i < \omega}$ when $\{b_{i}\}_{i < \omega}$ and $\{b'_{i}\}_{i < \omega}$ are Morley sequences in a stable theory.

    However, the conclusion of Theorem \ref{4-main4} does not hold.  Let $a_{1} \equiv_{M} a'_{1} \subseteq p_{1}(\mathbb{M})$, $ a_{2} \models p_{2}(\mathbb{M})$,  $ a_{3} \models p_{3}(\mathbb{M})$ with $\models a_{1} E a_{2}$, $\models a'_{1} E a_{3}$, $\models a_{2} E a_{3}$. Then $a_{1} \ind^{K^{*}}_{M} a_{2}$, $a'_{1} \ind^{K^{*}}_{M} a_{3}$, $a_{2} \ind^{K^{*}}_{M} a_{3}$, but $\mathrm{tp}(a_{1}/Ma_{2}) \cup \mathrm{tp}(a'_{1}/Ma_{3})$ is inconsistent.
\end{example}

\section{$\mathrm{NSOP}_{3}$ theories with symmetric Conant-independence}

In \cite{Sim20}, Simon proves the following independence theorem for $\mathrm{NTP}_{2}$ theories, using the independence theorem for $\mathrm{NTP}_{2}$ theories of Ben Yaacov and Chernikov (\cite{BYC07}).

\begin{fact}\label{4-indntp2}
Let $T$ be $\mathrm{NTP}_{2}$, and let $c \ind^{f}_{M} ab$ and $b \ind^{f}_{M} a$. Let $b' \equiv_{M} b$ with $b' \ind^{f}_{M} a$. Then there is some $c' \ind^{f}_{M} ab'$ with $c'a \equiv_{M} ca$ and $c'b' \equiv_{M} cb$.
\end{fact}

(In fact, Simon proves a more general version of this over extension bases.) He then poses the question

\begin{question}\label{4-simon?}
Suppose $p$ and $q$ are $M$-invariant types in an $\mathrm{NTP}_{2}$ theory with $p^{(\omega)}|_{M} = q^{(\omega)}|_{M}$, and let $B, C \supseteq M$ be small supersets of $M$. For some/every $B' \equiv_{M} B$ such that $B' \ind^{f}_{M} C$, is there $a \models p(x)|_{B'} \cup q(x)|_{C}$ with $a \ind^{f}_{M} B'C$?

\end{question}

This is true for simple theories by the independence theorem for simple theories (\cite{KP99}), and for NIP theories because $p^{(\omega)}(x)|_{M}$ determines any invariant type $p(x)$ (Proposition 2.36 of \cite{sim12}); Fact \ref{4-indntp2} justifies the equivalence of ``some" with ``any" $B'$. We show that a similar property holds for finitely satisfiable types in $\mathrm{NSOP}_{3}$ theories with symmetric Conant-independence:

\begin{theorem} \label{4-main5}
Let $T$ be an $\mathrm{NSOP}_{3}$ theory, and assume $\ind^{K^{*}}$ is symmetric. Suppose $p$ and $q$ are $M$-finitely satisfiable (global) types with $p^{(\omega)}|_{M} = q^{(\omega)}|_{M}$, and let $a, b \supseteq M$ be small supersets of $M$ with $a \ind^{K}_{M} b$. Then there is $c \models p(x)|_{a} \cup q(x)|_{b}$ with $c \ind^{K^{*}}_{M} ab$.

\end{theorem}

The ``some" part, the analogue of a positive answer to Question \ref{4-simon?}, will be supplied by the symmetry of Conant-independence. Then the ``every" part, corresponding to Fact \ref{4-indntp2}, will follow from $\mathrm{NSOP}_{3}$. 

We first study theories where Conant-independence is symmetric. Naïvely, one expects it to follow from compactness that $a\ind^{K^{*}}_{M}b$ implies the existence of an $Ma$-indiscernible $M$-invariant Morley sequence starting with $b$. This naïve argument fails, because the property of being an invariant Morley sequence of realizations of a fixed complete type over $M$ is not type-definable. However, the following proposition about theories with symmetric Conant-independence is enough for our purposes:

\begin{lemma}\label{4-symm5}
Suppose $\ind^{K^{*}}$ is symmetric, and let $I =\{a_{i}\}_{i \in \omega}$ be a coheir Morley sequence over $M$ with $a_{0}=a$ that is indiscernible over $Mb$. Then there is an $M$-invariant Morley sequence $J = \{b_{i}\}_{i\in \omega}$ with $b_{0}=b$ that is indiscernible over $Ma$.

\end{lemma}

\begin{proof}
The main claim of this proof is the following;

\begin{claim}
There exists $b'\equiv_{MI} b$ with $b' \ind_{M}^{i} b$ such that $I$ remains indiscernible over $bb'$.
\end{claim}

\begin{proof}

We first show that $b \ind_{M}^{K^{*}}I$. We need the following fact:

\begin{fact} \label{4-conantforkdivide} (Fact 6.1, \cite{GFA})
Let $\{c_{i}\}_{i \in \omega}$ be a coheir Morley sequence over $M$ with $c_{0}=c$ such that $\{\varphi(x, c_{i})\}_{i \in \omega}$ is consistent. Then $\varphi(x, c)$ does not Conant-fork over $M$.

\end{fact}

Now suppose $\varphi(x, \bar{a}) \in \mathrm{tp}(b/MI)$ for $\bar{a}=a_{0} \ldots a_{n}$. Then $\{\bar{a}_{i}\}_{i \in \omega}$ for $\bar{a}_{i}=a_{ni}\ldots a_{ni+(n-1)}$ is a finitely satisfiable Morley sequence over $M$ with $\bar{a}_{0}=\bar{a}$ such that $\{\varphi(x, \bar{a}_{i})\}_{i \in \omega}$ is consistent, so by the fact, $\varphi(x, \bar{a})$ does not Conant-fork over $M$ and $b \ind_{M}^{K^{*}}I$ is as desired. (See the proof of Proposition 5.3 of \cite{NSOP2}, or Proposition 3.21 of \cite{KR17}.)

Let $q(\bar{x}, b) =\mathrm{tp}(I/Mb)$. By symmetry, $I \ind^{K^{*}}_{M} b$, so for every $\varphi(\bar{x}, b) \in q(\bar{x}, b)$, there is some $b' \equiv_{M} b$ with $b'\ind^{i}_{M}b$ such that $\{\varphi(\bar{x}, b), \varphi(\bar{x}, b')\}$ is consistent. By compactness, the condition $x \ind^{i}_{M} b$ is type-definable over $Mb$ (contrast with the remark on invariant Morley sequences in the paragraph immediately preceding the proof of the proposition), so there is $b' \equiv_{M} b$ with $b'\ind^{i}_{M}b$ such that $q(\bar{x}, b) \cup q(\bar{x}, b')$ is consistent. By an automorphism, we can assume $b' \equiv_{MI} b$, and by Ramsey's theorem and compactness (and an automorphism), we can assume $I'$ is indiscernible over $bb'$.

\end{proof}

We now show by induction that we can find $b_{i}$ for $i < \kappa$, $\kappa$ large, such that $b = b_{0}$, $b_{i} \equiv_{MI} b$, $b_{i} \ind^{i}_{M} b_{< i}$, and $I$ is indiscernible over $Mb_{0} \ldots b_{\lambda}$ for $\lambda \leq  \kappa $. Suppose we have found $b_{i}$ for $ i < \lambda$ and we find $b_{\lambda}$: By the claim, there are $b'_{< \lambda} \equiv_{MI} b_{<\lambda}$ with $b'_{<\lambda} \ind^{i}_{M} b_{<\lambda}$ and $I$ indiscernible over $b_{<\lambda} b'_{\lambda}$. Now let $b_{\lambda} = b'_{0}$.

Then by the Erdős-Rado theorem, we can find an $MI$-indiscernible invariant Morley sequence sequence $J$ over $M$ starting with $b$, which will in particular be $Ma$-indiscernible.

\end{proof}

\begin{remark} \label{4-conantdiv}
Conant-forking is often equal to Conant-dividing at the level of formulas; for example if $\ind^{i}$ satisfies left extension, or $T$ has the \textit{strong witnessing property} that has no known counterexamples among the $\mathrm{NSOP}_{4}$ theories (Definition 3.5 of \cite{GFA}). In particular, we know of no theories where $\ind^{K^{*}}$ is symmetric and the relation $a \ind^{K^{*}d}_{M} b$, defined to hold when $\mathrm{tp}(a/Mb)$ has no Conant-dividing formulas, is not symmetric. If we assume the symmetry of $\ind^{K^{*}d}$ rather than $\ind^{K^{*}}$, we can prove Lemma \ref{4-symm5} for $I$ an invariant Morley sequence over $M$ rather than a coheir Morley sequence over $M$; the only difference is that we no longer use Fact \ref{4-conantforkdivide} on coheir Morley sequences. If we assume the symmetry of $\ind^{K^{*}d}$ rather than $\ind^{K^{*}}$ in Theorem \ref{4-main5}, we can then prove the conclusion when $p$ and $q$ are assumed to be $M$-invariant types rather than $M$-finitely satisfiable types and $\ind^{K^{*}}$ is replaced with $\ind^{K^{*}d}$, getting something closer to the claim of Simon in \cite{Sim20}; the proof will be exactly the same as the below, except Fact \ref{4-conantforkdivide} will not be used.

\end{remark}

\begin{lemma}\label{4-consistency5}
Assume $\ind^{K^{*}}$ is symmetric. Let $p$ and $q$ be $M$-finitely satisfiable types with $p^{(\omega)}|_{M} = q^{(\omega)}|_{M}$, and let $a, b \supseteq M$ be small sets containing $M$. Then there is some $M$- invariant type $r \vdash \mathrm{tp}(b/M)$ such that, for any $a_{1} \ldots a_{n}$ with $a_{i} \models \mathrm{tp}(a/M)$ for all $i < \omega$, and such that $b_{1} \ldots b_{m} \models r^{(m)}(y_{1}, \ldots, y_{m})|_{Ma_{1}\ldots a_{n}}$, $p(x)|_{a_1, \ldots, a_n}\cup \bigcup_{i=1}^{m} q(x)|_{b_{i}}$ is consistent.

\end{lemma}

\begin{proof}
    (See also the proof of Proposition 5.7 of \cite{NSOP2}, or Proposition 6.10, \cite{KR17}.) Let $I \models p^{(\omega)}|_{M} = q^{(\omega)}|_{M}$. By an automorphism, there is an $|M|+|a|$-saturated model $M' \succ M$ such that $I \models p^{(\omega)}|_{M'}$, and also by an automorphism, there is some $b' \equiv_{M}b$ with $I \models q^{(\omega)}|_{b'}$. By Ramsey's theorem and compactness, we can assume $I$ is indiscernible over $M'b'$; now let $c=c_{0}$ for $I = \{c_{i}\}_{i \in \omega}$. By Lemma \ref{4-symm5}, there is an $M$-invariant type $s(X, y) \vdash \mathrm{tp}(M'b'/M)$ and a Morley sequence $\{M'_{i}b'_{i}\}_{i \in \omega}$ with $M'_{0}b'_{0}=M'b'$ in $s(X,y)$ that is indiscernible over $Mc$. In particular, $p(x)|_{M'}\cup \bigcup^{m}_{i=1} q(x)|_{b'_{i}}$ is consistent, realized by $c$. Let $r(y) = s(X, y)|_{y}$; then for all $b''_{1} \ldots b''_{m} \models r^{(m)}(y_{1}, \ldots, y_{m})|_{M'}$, $p(x)|_{M'}\cup \bigcup^{m}_{i=1} q(x)|_{b''_{i}}$ is consistent. (This is because $b''_{1} \ldots b''_{m} \equiv_{M'} b'_{1} \ldots b'_{m}$: $ M'_{1}b'_{1} \ldots M'_{m}b'_{m} \models S^{(m)}(X_{1}, y_{1}, \ldots X_{m}, y_{m})|_{M'}$, so $b'_{1} \ldots b'_{m} \models S^{(m)}(X_{1}, y_{1}, \ldots X_{m}, y_{m})|_{M', y_{1}, \ldots. y_{m}})=r^{(m)}(y_{1}, \ldots, y_{m})|_{M'}$.) Let $a_{1}, \ldots, a_{n}$ have $a_{i} \models \mathrm{tp}(a/M)$ for $i \leq n$, and $b_{1}, \ldots b_{m} \models r^{(m)}(y_{1}, \ldots, y_{m})|_{Ma_{1}, \ldots, a_{n}}$. By $|M|+|a|$-saturation of $M'$, there are $a'_{1}, \ldots a'_{n} \in M'$ with $a'_{1}, \ldots a'_{n} \equiv_{M} a_{1}, \ldots a_{n}$. Let $b''_{1}, \ldots b''_{m} \models r^{(m)}(y_{1}, \ldots, y_{m})|_{M'}$. Then $p(x)|_{a'_1, \ldots, a'_n}\cup \bigcup^{m}_{i=1} q(x)|_{b''_{i}}$ is consistent. But by invariance, $a'_{1} \ldots a'_{n} b''_{1} \ldots b''_{n} \equiv_{M} a_{1} \ldots a_{n} b_{1} \ldots b_{n}$. So $p(x)|_{a_1, \ldots, a_n}\cup \bigcup^{m}_{i=1} q(x)|_{b_{i}}$ is consistent, as desired. 
\end{proof}

We are now ready to prove Theorem \ref{4-main5}.

\begin{proof}

First of all, replacing $p$ with $p^{(\omega)}$ and $q$ with $q^{(\omega)}$, we may assume that $p|_{M}=q|_{M}$ is the type of a coheir Morley sequence over $M$. Now assume $p(x)|_{a} \cup q(x)|_{b}$ is consistent, realized by a coheir Morley sequence $I$. It can be assumed indiscernible over $ab$ by Ramsey's theorem and compactness, so $ab \ind^{K^{*}}_{M} I$ by the paragraph immediately following Fact \ref{4-conantforkdivide}, and $I \ind^{K^{*}}_{M} ab$ by symmetry. So it suffices to show $p(x)|_{a} \cup q(x)|_{b}$ is consistent. Suppose otherwise, so there are $\varphi(x, a) \in p(x)|_{a} $ and $\psi(x, b) \in q(x)|_{b}$ such that $\{\varphi(x, a), \psi(x, b)\}$ is inconsistent. Let $s(w, y) = \mathrm{tp}(a, b/M)$, and let $r(y)$ be as in Lemma \ref{4-consistency5}. Once again, we use Conant's technique, Theorem 7.17 of \cite{Co15}. By induction, we will find $a_{1}, \ldots, a_{n}, \ldots, b_{1}, \ldots, b_{n}, \ldots$ such that

(1) For $i < j$, $a_{j} b_{i} \equiv_{M} ab$, so $\{\varphi(x, a_{j}), \psi(x, b_{i})\}$ is inconsistent and $a_{i} \models \mathrm{tp}(a/M)$ for each $i \geq 1$.

(2) For $n \leq m$, $b_{n+1}, \ldots b_{m} \models r^{(m-n)}(y)|_{Ma_{1}\ldots a_{n}}$, so $p(x)|_{a_1, \ldots, a_n}\cup \bigcup^{m}_{i=n+1} q(x)|_{b_{i}}$ and in particular $\{\varphi(x, a_{1}), \ldots \varphi(x, a_{n}), \psi(x, b_{n+1}) \ldots \psi(x, b_{m})\}$ is consistent by Lemma \ref{4-consistency5}.

Assume $a_{1}, \ldots, a_{n}, b_{1}, \ldots, b_{n}$ have already been constructed satisfying (1) and (2) up to $n$. Then  $b_{1}, \ldots, b_{n}$ begin an invariant Morley sequence in $\mathrm{tp}(b/M)$, so because $a \ind^{K}_{M} b$, $\cup_{i =1}^{n} s(w, b_{i})$ is consistent, and we can take $a_{n+1}$ to realize it. Then we can take $b_{n+1} \models r(y)|_{a_{1}, \ldots, a_{n+1}b_{1}\ldots b_{n}}$.

By \ref{4-sop3fact}, properties (1) and (2) imply $\mathrm{SOP}_{3}$--a contradiction. This proves Theorem \ref{4-main5}.

\end{proof}

However, as with the other two main theorems, Theorem \ref{4-main5} fails for $\mathrm{NSOP}_{4}$ theories with symmetric Conant-independence.

\begin{example}
The model companion $T$ of the theory of triangle-free graphs is $\mathrm{NSOP}_{4}$ and has symmetric Conant-independence; see \cite{GFA}. If $p$ is a nonalgebraic $M$-finitely satisfiable type, $p^{(\omega)}|_{M}$ is determined by $p|_{M}$: By  indiscernibility, $\neg x_{i} E x_{j} \in p^{(\omega)}(\bar{x})$ for $i < j$, as $ x_{i} E x_{j} \in p^{(\omega)}(\bar{x})$ for all $ i < j$ is impossible.

Next, we claim that, if $M$ is countable, for $p_{0}(x) \in S_{1}(M)$ the complete type over $M$ containing $\neg x E m$ for all $m \in M$, there are $M$-finitely satisfiable types $p_{1}$ and $p_{2}$ extending $p_{0}(x)$, and $x E b_{i} \in p_{i}(x)$ for $i =1,2$, $b_{i} \notin M$, such that there is no $m \in M$ with $b_{1} E m \wedge b_{2} E m$. Let $\{S_{i}\}_{i \in \omega}$ enumerate the set $F$ of subsets of $M$ defined by $M$-formulas in $p_{0}(x)$. We choose, by induction, disjoint anticliques $A$, $B$ of $M$, both of which meet each of the $S_{i}$. Namely, we construct disjoint anticliques $A_{n}, B_{n}$ for $n \in \omega$, such that $A_{n} \cap S_{i} \neq \emptyset$ and $B_{n} \cap S_{i} \neq \emptyset$ for $i \leq n$ and $A_{i} \subseteq A_{j}$ and $B_{i} \subseteq B_{j}$ for $i \leq j$, and take $A = \cup^{\infty}_{i = 0} A_{i}$ and $B = \cup^{\infty}_{i = 0} B_{i}$. Suppose $A_{n}, B_{n}$ already constructed. Since $S_{n+1}$ contains a set defined by a conjunction of formulas of the form $x \neq m$ and $\neg xEm$ for $m \in M$, and $M$ is a model of the model companion of the theory of triangle-free graphs, we can find distinct $a_{n+1}, b_{n+1} \in S_{n+1} \backslash A_{n} \cup B_{n}$ such that $\neg a_{n+1} E a$ for any $a \in A_{n}$, $\neg b_{n+1} E b$ for $b \in B_{n}$, and take $A_{n+1} = A_{n} \cup \{a_{n+1}\}$ and $B_{n+1} = B_{n} \cup \{b_{n+1}\}$. Now let $U_{1}$ be an ultrafilter containing $F \cup \{A\}$ and $U_{2}$ be an ultrafilter containing $F \cup \{B\}$. Let $p_{i}(x)= \{\varphi(x, b) \in L(\mathbb{M}): \varphi(M, b) \in U_{i}  \}$. Let $b_{1} \in \mathbb{M}$ be such that, for $m \in M$, $b_{1} E m $ if and only if $m \in A$, and similarly for $b_{2}$ and $B$. This is possible because $A$ and $B$ are anticliques. Then $p_{1}$, $p_{2}$, $b_{1}$, $b_{2}$ are as desired in the claim.

There is an invariant type $q$ extending $\mathrm{tp}(b_{1}/M)$ such that, for $b'_{1}\models q(x)|_{Mb_{2}}$, $b'_{1} E b_{2}$; for example, we can require that $x E b \in q(x)$ if and only if $x E b \in \mathrm{tp}(b_{1}/M)$  or $b \models \mathrm{tp}(b_{2}/M)$. This gives a consistent type: let $a_{*}$ be a node satisfying these relations in a graph extending $\mathbb{M}$; then there are no triangles involving $a_{*}$, a realization of $\mathrm{tp}(b_{2}/M)$ in $\mathbb{M}$, and an element of $M$, because we chose $A$ and $B$ to be disjoint; there are also no edges between realizations of $\mathrm{tp}(b_{2}/M)$ in $\mathbb{M}$, because $B$ is nonempty and there are no triangles in $\mathbb{M}$, so there are no triangles involving $a_{*}$ and two realizations of $\mathrm{tp}(b_{2}/M)$ in $\mathbb{M}$. Let $b'_{1} \models q(x)|_{Mb_{2}}$; then $b'_{1} \ind^{K} b_{2}$ and $b'_{1} E b_{2}$.

But $p_{1}(x)|Mb'_{1} \cup p_{2}(x)|Mb_{2}$ is inconsistent.
\end{example}

In the proof of Theorem \ref{4-main5} above, symmetry of $\ind^{K^{*}}$ is not used directly in building the configuration satisfying (1) and (2); this is in contrast to \cite{NSOP2}, where the rows are required to be (coheir) Conant-independent throughout the construction. We now prove a version of Theorem \ref{4-main5} for Kim-nonforking types over $M$ rather than finitely satisfiable or invariant types over $M$, that uses the full force of the assumption that the relevant independence relation, in this case $\ind^{K^{*}d}$, is symmetric.

By Remark \ref{4-conantdiv}, in an $\mathrm{NSOP}_{3}$ theory where Conant-forking coincides with Conant-dividing and $\ind^{K^{*}}$ is symmetric, Theorem \ref{4-main5} holds even if $p$ and $q$ are only assumed to be $M$-invariant types with $p^{(\omega)}|_{M} = q^{(\omega)}|_{M}$, rather than $M$-finitely satisfiable types. In this case, $p$ and $q$ are examples of types, such that for any small $A,B \supseteq M$, there are $M$-invariant Morley sequences $I=\{a_{i}\}_{i \in \omega}$ and $I' = \{b_{j}\}_{j \in \omega}$, such that $a_{i} \models p(x)|_{A}$ and $b_{i} \models q(x)|_{B}$ for $i \geq 0$, $I \equiv_{M} I'$, and $I \ind^{K}_{M} A$ (and $I' \ind^{K}_{M} B$). This assumption can be seen as an analogue of $p^{(\omega)}|_{M} = q^{(\omega)}|_{M}$ for Kim-nonforking types over $M$, and yields the conclusion of Theorem \ref{4-main5} with respect to $\ind^{i}$:

\begin{theorem}\label{4-main5b}
    Assume $\ind^{K^{*}d}$ is symmetric and $T$ is $\mathrm{NSOP}_{3}$. Let $p(x)$ be an $M$-invariant type, $a, b \supseteq M$ be small supersets of $M$ with $a \ind_{M}^{i} b$ and $I$, $J$ $M$-invariant Morley sequences in $p(x)$ indiscernible over $a$ and $b$ respectively, with $I \ind_{M}^{K} a$. Then there is some $I'' \ind_{M}^{K^{*}d} ab$ with $I'' \equiv_{a} I$ and $I'' \equiv_{b} I'$. If $\ind^{f}$ (resp. $\ind^{K}$) satisfies the chain condition, the assumption $a \ind_{M}^{i} b$ can be relaxed to $a \ind_{M}^{f} b$ (resp. $a \ind_{M}^{K} b$).
\end{theorem}

Note that $\ind^{f}$ is known to satisfy the chain condition in $\mathrm{NTP}_{2}$ theories (Proposition 2.8, \cite{BYC07}). It is not known whether there are non-simple examples of $\mathrm{NSOP}_{3}$ $\mathrm{NTP}_{2}$ theories. (Problem 3.16, \cite{Che14}).

We start with the analogue of Lemma \ref{4-consistency5}.

\begin{lemma}\label{4-consistency5b}
 Let $M$, $I$, $J$, $a$, $b$, be as in the statement of Theorem \ref{4-main5b}, and assume $\ind^{K^{*}d}$ is symmetric. Let  $p(X, y) = \mathrm{tp}(I, a/M)$ and $q(X, z) = \mathrm{tp}(J, b/M)$. Then there is some invariant type $r \models \mathrm{tp}(b/M)$ such that for $a_{1}, \ldots, a_{n}$ with $a_{i} \models \mathrm{tp}(a/M)$ for $i < n$ beginning an invariant Morley sequence over $M$ and $b_{1}, \ldots b_{m} \models r^{(m)}(y)|_{Ma_{1}\ldots a_{n}}$, $\cup_{i=1}^{n}p(X, a_{i}) \cup \bigcup_{i =1}^{m} q(X, b_{i})$ is consistent.
\end{lemma}

\begin{proof}
    Let $\{K^{i}\}_{i < \kappa}$ enumerate the invariant Morley sequences in $\mathrm{tp}(a/M)$. Since $I \ind_{M}^{K} a$, $\{p(X, a_{i})\}_{i \in \omega}$ is consistent for $\{a_{i}\}$ any invariant Morley sequence in $\mathrm{tp}(a/M)$, so by automorphisms, there are $\{K'^{i}\}_{i < \kappa}$ such that for $i< \kappa$ and $K'^{i}=\{a'^{i}_{j}\}_{j \in \omega}$, for $j \in \omega$, $a'^{i}_{j} \equiv_{MI} a$. Let $K'= \cup_{i < \kappa} K'^{i}$. By another automorphism, find $b'$ with $b'I \equiv bJ$. Then by Ramsey's theorem and compactness, $I$ can be assumed indiscernible over $K'b'$. By Remark \ref{4-conantdiv}, there is an $M$-invariant type $s(X, y) \vdash \mathrm{tp}(K'b'/M)$ and a Morley sequence $\{K'_{i}b'_{i}\}_{i < \omega }$ with $K'_{0}b'_{0}=K'b'$ in $s(X,y)$ that is indiscernible over $MI$.

    In particular, for any $i < \kappa$, $\cup_{j < \omega} p(X, a'^{i}_{j}) \cup \bigcup_{j < \omega} q(X, b'_{j})$ is consistent, realized by $I$. Let $r(y) = s(X, y)|_{y}$; then for $b''_{1}, \ldots b''_{m} \models r^{(m)}(y)|_{K'}$, for any $i < \kappa$, $\cup_{j < \omega} p(X, a'^{i}_{j}) \cup \bigcup_{j \leq m} q(X, b''_{j})$ is consistent. (As in the proof of Lemma \ref{4-consistency5}), $b''_{1} \ldots b''_{m} \equiv_{K'} b'_{1} \ldots b'_{m}$: $ K'_{1}b'_{1} \ldots K'_{m}b'_{m} \models S^{(m)}(X_{1}, y_{1}, \ldots X_{m}, y_{m})|_{K'}$, so $b'_{1} \ldots b'_{m} \models S^{(m)}(X_{1}, y_{1}, \ldots X_{m}, y_{m})|_{K', y_{1}, \ldots. y_{m}})=r^{(m)}(y_{1}, \ldots, y_{m})|_{K'}$.) Let $a_{1}, \ldots, a_{n}$ begin an invariant Morley sequence over $M$ with $a_{i} \models \mathrm{tp}(a/M)$ for $i \leq n$, and $b_{1}, \ldots b_{m} \models r^{(m)}(y)|_{Ma_{1}, \ldots, a_{n}}$. Then there are $a'_{1}, \ldots a'_{n} \in K'$ with $a'_{1}, \ldots a'_{n} \equiv_{M} a_{1}, \ldots a_{n}$. Let $b''_{1}, \ldots b''_{m} \models r^{(m)}(y)|_{K'}$. Then $\cup_{j \leq n} p(X, a'_{j}) \cup \bigcup_{j \leq m} q(X, b''_{j})$ is consistent. But by invariance, $a'_{1} \ldots a'_{n} b''_{1} \ldots b''_{n} \equiv_{M} a_{1} \ldots a_{n} b_{1} \ldots b_{n}$. So $\cup_{j \leq n} p(X, a_{j}) \cup \bigcup_{j \leq m} q(X, b_{j})$  is consistent, as desired. 
\end{proof}

We fix the auxilliary notation $a\ind_{M}^{K^{+}}b$ to mean that there is an $M$-invariant Morley sequence $J = \{b_{i}\}_{i\in \omega}$ with $b_{0}=b$ that is indiscernible over $Ma$. By Remark \ref{4-conantdiv}, if $\ind^{K^{*}d}$ is symmetric then so is $\ind^{K^{+}}$. We prove the following lemma about $\ind^{K^{+}}$ and $\ind^{K}$:

\begin{lemma}\label{4-wit5}
   Let $d_{0} \ind^{K^{+}}_{M} c$ and $d_{1} \ind^{K}_{M} c$. Then there is $d'_{1} \equiv_{Mc} d_{1}$ with $d_{0}d'_{1} \ind^{K^{+}} c$ and $d'_{1} \ind^{i}_{M} d_{0}$.
\end{lemma}

\begin{proof}

    (See also the proof of Proposition 5.7 of \cite{NSOP2}, or Proposition 6.10, \cite{KR17}.) 
    By $d_{0} \ind^{K^{+}}_{M} c$, let $I=\{c_{i}\}_{i \in \omega}$ with $c_{0} = c$ be an $M$-invariant Morley sequence indiscernible over $d_{0}$. By $d_{1} \ind^{K}_{M} c$ and compactness, there is some $d''_{1}$ with $d''_{1} c_{i} \equiv_{M} d_{1}c$ for $i < \omega$. By Ramsey's theorem, compactness, and an automorphism, we can choose $d''_{1}$ so that $I$ is indiscernible over $Md_{0}d''_{1}$, so $d_{0}d''_{1}\ind^{K^{+}}_{M} I$: if $I$ is a Morley sequence in the $M$-invariant type $s$, then by compactness, there is a Morley sequence $\{I_{i}\}_{i < \omega}$ in $s^{(\omega)}$ over $M$ with $I_{0} = I$ that is indiscernible over $Md_{0}d''_{1}$ (See also the paragraph immediately following Fact \ref{4-conantforkdivide}). So by the paragraph immediately preceding the statement of the lemma, $I \ind^{K^{+}}_{M} d_{0}d''_{1}$, and in particular there is $d'_{0}d'_{1} \ind^{i}_{M} d_{0}d''_{1}$ with $d'_{0}d'_{1} \equiv_{MI} d_{0}d''_{1}$; by Ramsey's theorem, compactness, and an automorphism, we can choose $d'_{0}d'_{1}$ so that $I$ is indiscernible over $Md'_{0}d'_{1}d_{0}d''_{1}$. Then, again, $d'_{0}d'_{1}d_{0}d''_{1} \ind^{K^{+}}_{M} I$, so in particular,  $d_{0}d'_{1} \ind^{K^{+}} c$; also, $d_{1}' \equiv_{Mc} d''_{1} \equiv_{Mc} d_{1}$.
\end{proof}

We are now ready to prove Theorem \ref{4-main5b}. Note that if we can find $I''$ so that $I'' \equiv_{a} I$ and $I'' \equiv_{b} I'$, then $I''$ can be chosen indiscernible over $ab$, so $ab \ind^{K^{*}d}_{M}I''$, and $I'' \ind^{K^{*}d}_{M} ab$. So it suffices to show $p(X, a) \cup q(X, b)$ is consistent. Suppose it is inconsistent. Then by compactness, there are some $\varphi(X, a) \in p(X, a)$ and $\psi(X, b) \in q(X, b)$ such that $\{\varphi(X, a), \psi(X, b)\}$ is inconsistent. Let $s(w, y) = \mathrm{tp} (a,b/M)$ and let $r(y)$ be as in Lemma 5.1. Let $\kappa$ be large. By transfinite induction, we will find $a_{i}, b_{i}$, $i < \kappa$ such that

(1) For $i < j < \kappa$, $a_{j} b_{i} \equiv_{M} ab$, so $\{\varphi(x, a_{j}), \psi(x, b_{i})\}$ is inconsistent and $a_{i} \models \mathrm{tp}(a/M)$.

(2) For $i < j_{1} < \ldots < j_{m} <\kappa $, $b_{j_{1}}, \ldots b_{j_{m}} \models r^{(m)}(y)|_{Ma_{\leq i}}$ and $a_{i} \ind_{M}^{i} a_{<i}$.

(3) For $i < \kappa$, $a_{\leq i} \ind_{M}^{K^{+}} b_{\leq i}$.

Suppose $a_{i}, b_{i}$ already constructed satisfying (1)-(3) for $i < \lambda$. We find $a_{\lambda}$ and $b_{\lambda}$. By $a \ind^{i}_{M} b$, or $a \ind^{f}_{M} b$ or $a \ind^{K}_{M} b$ and the respective chain condition, since $\{b_{i}\}_{i < \lambda}$ is an invariant Morley sequence in $\mathrm{tp}(b/M)$, there is some $a_{\lambda} \models \cup_{i< \lambda} s(w, b_{i})$ with $a_{\lambda} \ind^{K}_{M} b_{<\lambda}$. By Lemma \ref{4-wit5}, $a_{\lambda}$ can then additionally be chosen with $a_{\lambda} \ind^{i}_{M} a_{< \lambda}$ and $a_{\lambda}a_{< \lambda} \ind^{K^{+}}_{M} b_{< \lambda}$, as desired. We then choose $b_{\lambda} \models r(y)|_{a_{\leq \lambda}b_{<\lambda}}$, which will preserve (1) and (2); it remains to show (3). This will follow from the following claim, analogous to Claim 6.2 of \cite{NSOP2}:

\begin{claim}\label{4-inv5}
For any $a, b, c$, $M$, if $a \ind^{K^{+}}_{M} b$ and $\mathrm{tp}(c/Mab)$ extends to an $M$-invariant type $q(x)$, then $a \ind^{K^{+}}_{M} bc$. (This is true as long as $\ind^{K^{+}}$ is symmetric.)
\end{claim}

\begin{proof}
It follows that $b \ind^{K^{+}}_{M} a$, so let $I=\{a_{i}\}_{i \in \omega}$ be an $Mb$-indiscernible invariant Morley sequence over $M$ with $a_{0} = a$. By an automorphism, we can choose $I$ so that $c \models q|MIb$. By Ramsey and compactness, we can further choose $I$ indiscernible over $Mbc$, so $bc \ind^{K^{+}}_{M} a$ and by symmetry of $\ind^{K^{+}}$, $a \ind^{K^{+}} bc$. 
\end{proof}

Finally, by the Erdős-Rado theorem, we can find $\{a_{i}b_{i}\}_{i < \omega}$ indiscernible over $M$, satisfying (1) and (2) (and (3)). Then $\{a_{i}\}_{i \in \omega}$ will be an invariant Morley sequence over $M$ with $a_{i} \models \mathrm{tp}(a/M)$, so for $n \leq m$, by (2) and Lemma \ref{4-consistency5b}, $\cup_{i=1}^{n}p(X, a_{i}) \cup \cup_{i =n+1}^{m} q(X, b_{i})$ and therefore $\{\varphi(x, a_{1}), \ldots \varphi(x, a_{n}), \psi(x, b_{n+1}) \ldots \psi(x, b_{m})\}$ is consistent. This, together with (1), implies $\mathrm{SOP}_{3}$ by Fact \ref{4-sop3fact} --a contradiction. 

\:

\textbf{Acknowledgements} The author would like to thank Hyoyoon Lee, Byunghan Kim, and the other participants of the Yonsei University logic seminar for their comments on the degree of Kim-dividing $k=2$, and Nicholas Ramsey for bringing to the author's attention Question \ref{4-simon?} of Simon. The author would also like to thank Itay Kaplan for enlightening discussions on the results of this paper, which brought to light an error in the statement of Lemma \ref{4-consistency5} in a previous draft, as well as the referee for many helpful comments and suggestions.

\bibliographystyle{plain}
\bibliography{refs}

\end{document}